\documentclass[runningheads,a4paper,envcountsame,envcountsect]{llncs}

\usepackage[utf8]{inputenc}
\usepackage{xr}
\usepackage{mathtools}

\let\saveendproof=\endproof
\def\endproof{\qed\saveendproof}

\newtheorem{openquestion}[theorem]{Open question}

\usepackage{booktabs}
\usepackage{array,longtable,booktabs}
\newcolumntype{C}{>{$}c<{$}}  
\newcolumntype{L}{>{$}l<{$}}  
\usepackage{pdflscape} 

\usepackage{amssymb}
\usepackage{graphicx}
\usepackage{epstopdf}
\usepackage[hyphens]{url}
\usepackage[usenames,svgnames,table]{xcolor}

\usepackage{multirow}

\usepackage{tikz}
\definecolor{mediumspringgreen}{rgb}{0.0, 0.98039215, 0.60392156}

\usepackage{pgfplots} 

\usetikzlibrary{positioning}
\def\visible<#1>{}  
\usepackage[utf8]{inputenc}

\usepackage[english]{babel}
\usepackage{amsfonts}
\usepackage{amsmath}
\usepackage{latexsym}
\usepackage{color}
\usepackage{subfigure}
\usepackage{enumerate}

\usepackage{hyperref}  

\usepackage{ifpdf}
\usepackage{listings}
\DeclareMathOperator    \conv           {conv}

\DeclareMathOperator    \intr                   {int}

\DeclareMathOperator    \relint         {rel\,int}
\DeclareMathOperator    \verts          {vert}

\providecommand\compactop[1]{\kern0.2pt{#1}\kern0.2pt\relax}

\newcommand{\bb}{\mathbb}
\newcommand{\R}{\bb R}
\newcommand{\Q}{\bb Q}
\newcommand{\Z}{\bb Z}

\newcommand\st{\mid}
\newcommand\bigst{\mathrel{\big|}}
\newcommand\Bigst{\mathrel{\Big|}}

\renewcommand{\P}{\mathcal{P}}

\def\st{\mid}
\newenvironment{psmallmatrixbig}{\bigl(\smallmatrix}{\endsmallmatrix\bigr)}
\newcommand\InlineFrac[2]{#1/#2}  
\newcommand\ColVec[3][\relax]
{
  \ifx#1\relax
  \bgroup\let\frac=\InlineFrac\begin{psmallmatrixbig}#2\vphantom{/}\\#3\vphantom{/}\end{psmallmatrixbig}\egroup
  \else
  \bgroup\let\frac=\InlineFrac\begin{psmallmatrixbig}\ifx#200\else#2/#1\fi\\\ifx#300\else#3/#1\fi\end{psmallmatrixbig}\egroup
  \fi
}

\colorlet{validcolor}{ForestGreen} 
\colorlet{minimalcolor}{yellow}  
\colorlet{extremecolor}{red} 
\colorlet{facetcolor}{blue}  
\colorlet{weakfacetcolor}{orange} 

\colorlet{additivecolor}{MediumSpringGreen}

\def\valid/{\textcolor{validcolor}{valid}}
\def\minimal/{\textcolor{minimalcolor}{minimal}}
\def\extreme/{\textcolor{extremecolor}{extreme}}

\def\Ptight(#1){\textcolor{tightsolutionscolor}{P}(#1)}

\def\Eadditive(#1){\colorbox{additivecolor}{$E$}(#1)}

\newcommand\pilifted{\hat{\pi}}  

\makeatletter
\renewcommand{\pod}[1]
{\allowbreak\mathchoice{\mkern18mu}{\mkern8mu}{\mkern8mu}{\mkern8mu}(#1)}
\makeatother

\DeclareRobustCommand\sage[1]{\texttt{#1}}

\let\Myunderscore=\textunderscore   
\newcommand\underscore{\Myunderscore\allowbreak}
\let\_=\underscore

\newcommand\githubsearchurl{https://github.com/mkoeppe/cutgeneratingfunctionology/search}
\pgfkeyssetvalue{/sagefunc/bccz_counterexample}{\href{\githubsearchurl?q=\%22def+bccz_counterexample(\%22}{\sage{bccz\underscore{}counterexample}}}
\pgfkeyssetvalue{/sagefunc/bcdsp_arbitrary_slope}{\href{\githubsearchurl?q=\%22def+bcdsp_arbitrary_slope(\%22}{\sage{bcdsp\underscore{}arbitrary\underscore{}slope}}}
\pgfkeyssetvalue{/sagefunc/bhk_gmi_irrational}{\href{\githubsearchurl?q=\%22def+bhk_gmi_irrational(\%22}{\sage{bhk\underscore{}gmi\underscore{}irrational}}}
\pgfkeyssetvalue{/sagefunc/bhk_irrational}{\href{\githubsearchurl?q=\%22def+bhk_irrational(\%22}{\sage{bhk\underscore{}irrational}}}
\pgfkeyssetvalue{/sagefunc/bhk_slant_irrational}{\href{\githubsearchurl?q=\%22def+bhk_slant_irrational(\%22}{\sage{bhk\underscore{}slant\underscore{}irrational}}}
\pgfkeyssetvalue{/sagefunc/chen_4_slope}{\href{\githubsearchurl?q=\%22def+chen_4_slope(\%22}{\sage{chen\underscore{}4\underscore{}slope}}}
\pgfkeyssetvalue{/sagefunc/dg_2_step_mir}{\href{\githubsearchurl?q=\%22def+dg_2_step_mir(\%22}{\sage{dg\underscore{}2\underscore{}step\underscore{}mir}}}
\pgfkeyssetvalue{/sagefunc/dg_2_step_mir_limit}{\href{\githubsearchurl?q=\%22def+dg_2_step_mir_limit(\%22}{\sage{dg\underscore{}2\underscore{}step\underscore{}mir\underscore{}limit}}}
\pgfkeyssetvalue{/sagefunc/drlm_2_slope_limit}{\href{\githubsearchurl?q=\%22def+drlm_2_slope_limit(\%22}{\sage{drlm\underscore{}2\underscore{}slope\underscore{}limit}}}
\pgfkeyssetvalue{/sagefunc/drlm_3_slope_limit}{\href{\githubsearchurl?q=\%22def+drlm_3_slope_limit(\%22}{\sage{drlm\underscore{}3\underscore{}slope\underscore{}limit}}}
\pgfkeyssetvalue{/sagefunc/drlm_backward_3_slope}{\href{\githubsearchurl?q=\%22def+drlm_backward_3_slope(\%22}{\sage{drlm\underscore{}backward\underscore{}3\underscore{}slope}}}
\pgfkeyssetvalue{/sagefunc/gj_2_slope}{\href{\githubsearchurl?q=\%22def+gj_2_slope(\%22}{\sage{gj\underscore{}2\underscore{}slope}}}
\pgfkeyssetvalue{/sagefunc/gj_2_slope_repeat}{\href{\githubsearchurl?q=\%22def+gj_2_slope_repeat(\%22}{\sage{gj\underscore{}2\underscore{}slope\underscore{}repeat}}}
\pgfkeyssetvalue{/sagefunc/gj_forward_3_slope}{\href{\githubsearchurl?q=\%22def+gj_forward_3_slope(\%22}{\sage{gj\underscore{}forward\underscore{}3\underscore{}slope}}}
\pgfkeyssetvalue{/sagefunc/gmic}{\href{\githubsearchurl?q=\%22def+gmic(\%22}{\sage{gmic}}}
\pgfkeyssetvalue{/sagefunc/hildebrand_2_sided_discont_1_slope_1}{\href{\githubsearchurl?q=\%22def+hildebrand_2_sided_discont_1_slope_1(\%22}{\sage{hildebrand\underscore{}2\underscore{}sided\underscore{}discont\underscore{}1\underscore{}slope\underscore{}1}}}
\pgfkeyssetvalue{/sagefunc/hildebrand_2_sided_discont_2_slope_1}{\href{\githubsearchurl?q=\%22def+hildebrand_2_sided_discont_2_slope_1(\%22}{\sage{hildebrand\underscore{}2\underscore{}sided\underscore{}discont\underscore{}2\underscore{}slope\underscore{}1}}}
\pgfkeyssetvalue{/sagefunc/hildebrand_5_slope_22_1}{\href{\githubsearchurl?q=\%22def+hildebrand_5_slope_22_1(\%22}{\sage{hildebrand\underscore{}5\underscore{}slope\underscore{}22\underscore{}1}}}
\pgfkeyssetvalue{/sagefunc/hildebrand_5_slope_24_1}{\href{\githubsearchurl?q=\%22def+hildebrand_5_slope_24_1(\%22}{\sage{hildebrand\underscore{}5\underscore{}slope\underscore{}24\underscore{}1}}}
\pgfkeyssetvalue{/sagefunc/hildebrand_5_slope_28_1}{\href{\githubsearchurl?q=\%22def+hildebrand_5_slope_28_1(\%22}{\sage{hildebrand\underscore{}5\underscore{}slope\underscore{}28\underscore{}1}}}
\pgfkeyssetvalue{/sagefunc/hildebrand_discont_3_slope_1}{\href{\githubsearchurl?q=\%22def+hildebrand_discont_3_slope_1(\%22}{\sage{hildebrand\underscore{}discont\underscore{}3\underscore{}slope\underscore{}1}}}
\pgfkeyssetvalue{/sagefunc/kf_n_step_mir}{\href{\githubsearchurl?q=\%22def+kf_n_step_mir(\%22}{\sage{kf\underscore{}n\underscore{}step\underscore{}mir}}}
\pgfkeyssetvalue{/sagefunc/kzh_10_slope_1}{\href{\githubsearchurl?q=\%22def+kzh_10_slope_1(\%22}{\sage{kzh\underscore{}10\underscore{}slope\underscore{}1}}}
\pgfkeyssetvalue{/sagefunc/kzh_28_slope_1}{\href{\githubsearchurl?q=\%22def+kzh_28_slope_1(\%22}{\sage{kzh\underscore{}28\underscore{}slope\underscore{}1}}}
\pgfkeyssetvalue{/sagefunc/kzh_28_slope_2}{\href{\githubsearchurl?q=\%22def+kzh_28_slope_2(\%22}{\sage{kzh\underscore{}28\underscore{}slope\underscore{}2}}}
\pgfkeyssetvalue{/sagefunc/kzh_3_slope_param_extreme_1}{\href{\githubsearchurl?q=\%22def+kzh_3_slope_param_extreme_1(\%22}{\sage{kzh\underscore{}3\underscore{}slope\underscore{}param\underscore{}extreme\underscore{}1}}}
\pgfkeyssetvalue{/sagefunc/kzh_3_slope_param_extreme_2}{\href{\githubsearchurl?q=\%22def+kzh_3_slope_param_extreme_2(\%22}{\sage{kzh\underscore{}3\underscore{}slope\underscore{}param\underscore{}extreme\underscore{}2}}}
\pgfkeyssetvalue{/sagefunc/kzh_4_slope_param_extreme_1}{\href{\githubsearchurl?q=\%22def+kzh_4_slope_param_extreme_1(\%22}{\sage{kzh\underscore{}4\underscore{}slope\underscore{}param\underscore{}extreme\underscore{}1}}}
\pgfkeyssetvalue{/sagefunc/kzh_5_slope_fulldim_1}{\href{\githubsearchurl?q=\%22def+kzh_5_slope_fulldim_1(\%22}{\sage{kzh\underscore{}5\underscore{}slope\underscore{}fulldim\underscore{}1}}}
\pgfkeyssetvalue{/sagefunc/kzh_5_slope_fulldim_2}{\href{\githubsearchurl?q=\%22def+kzh_5_slope_fulldim_2(\%22}{\sage{kzh\underscore{}5\underscore{}slope\underscore{}fulldim\underscore{}2}}}
\pgfkeyssetvalue{/sagefunc/kzh_5_slope_fulldim_3}{\href{\githubsearchurl?q=\%22def+kzh_5_slope_fulldim_3(\%22}{\sage{kzh\underscore{}5\underscore{}slope\underscore{}fulldim\underscore{}3}}}
\pgfkeyssetvalue{/sagefunc/kzh_5_slope_fulldim_4}{\href{\githubsearchurl?q=\%22def+kzh_5_slope_fulldim_4(\%22}{\sage{kzh\underscore{}5\underscore{}slope\underscore{}fulldim\underscore{}4}}}
\pgfkeyssetvalue{/sagefunc/kzh_5_slope_fulldim_5}{\href{\githubsearchurl?q=\%22def+kzh_5_slope_fulldim_5(\%22}{\sage{kzh\underscore{}5\underscore{}slope\underscore{}fulldim\underscore{}5}}}
\pgfkeyssetvalue{/sagefunc/kzh_5_slope_fulldim_covers_1}{\href{\githubsearchurl?q=\%22def+kzh_5_slope_fulldim_covers_1(\%22}{\sage{kzh\underscore{}5\underscore{}slope\underscore{}fulldim\underscore{}covers\underscore{}1}}}
\pgfkeyssetvalue{/sagefunc/kzh_5_slope_fulldim_covers_2}{\href{\githubsearchurl?q=\%22def+kzh_5_slope_fulldim_covers_2(\%22}{\sage{kzh\underscore{}5\underscore{}slope\underscore{}fulldim\underscore{}covers\underscore{}2}}}
\pgfkeyssetvalue{/sagefunc/kzh_5_slope_fulldim_covers_3}{\href{\githubsearchurl?q=\%22def+kzh_5_slope_fulldim_covers_3(\%22}{\sage{kzh\underscore{}5\underscore{}slope\underscore{}fulldim\underscore{}covers\underscore{}3}}}
\pgfkeyssetvalue{/sagefunc/kzh_5_slope_fulldim_covers_4}{\href{\githubsearchurl?q=\%22def+kzh_5_slope_fulldim_covers_4(\%22}{\sage{kzh\underscore{}5\underscore{}slope\underscore{}fulldim\underscore{}covers\underscore{}4}}}
\pgfkeyssetvalue{/sagefunc/kzh_5_slope_fulldim_covers_5}{\href{\githubsearchurl?q=\%22def+kzh_5_slope_fulldim_covers_5(\%22}{\sage{kzh\underscore{}5\underscore{}slope\underscore{}fulldim\underscore{}covers\underscore{}5}}}
\pgfkeyssetvalue{/sagefunc/kzh_5_slope_fulldim_covers_6}{\href{\githubsearchurl?q=\%22def+kzh_5_slope_fulldim_covers_6(\%22}{\sage{kzh\underscore{}5\underscore{}slope\underscore{}fulldim\underscore{}covers\underscore{}6}}}
\pgfkeyssetvalue{/sagefunc/kzh_5_slope_q22_f10_1}{\href{\githubsearchurl?q=\%22def+kzh_5_slope_q22_f10_1(\%22}{\sage{kzh\underscore{}5\underscore{}slope\underscore{}q22\underscore{}f10\underscore{}1}}}
\pgfkeyssetvalue{/sagefunc/kzh_5_slope_q22_f10_2}{\href{\githubsearchurl?q=\%22def+kzh_5_slope_q22_f10_2(\%22}{\sage{kzh\underscore{}5\underscore{}slope\underscore{}q22\underscore{}f10\underscore{}2}}}
\pgfkeyssetvalue{/sagefunc/kzh_5_slope_q22_f10_3}{\href{\githubsearchurl?q=\%22def+kzh_5_slope_q22_f10_3(\%22}{\sage{kzh\underscore{}5\underscore{}slope\underscore{}q22\underscore{}f10\underscore{}3}}}
\pgfkeyssetvalue{/sagefunc/kzh_5_slope_q22_f10_4}{\href{\githubsearchurl?q=\%22def+kzh_5_slope_q22_f10_4(\%22}{\sage{kzh\underscore{}5\underscore{}slope\underscore{}q22\underscore{}f10\underscore{}4}}}
\pgfkeyssetvalue{/sagefunc/kzh_5_slope_q22_f2_1}{\href{\githubsearchurl?q=\%22def+kzh_5_slope_q22_f2_1(\%22}{\sage{kzh\underscore{}5\underscore{}slope\underscore{}q22\underscore{}f2\underscore{}1}}}
\pgfkeyssetvalue{/sagefunc/kzh_6_slope_1}{\href{\githubsearchurl?q=\%22def+kzh_6_slope_1(\%22}{\sage{kzh\underscore{}6\underscore{}slope\underscore{}1}}}
\pgfkeyssetvalue{/sagefunc/kzh_6_slope_fulldim_covers_1}{\href{\githubsearchurl?q=\%22def+kzh_6_slope_fulldim_covers_1(\%22}{\sage{kzh\underscore{}6\underscore{}slope\underscore{}fulldim\underscore{}covers\underscore{}1}}}
\pgfkeyssetvalue{/sagefunc/kzh_6_slope_fulldim_covers_2}{\href{\githubsearchurl?q=\%22def+kzh_6_slope_fulldim_covers_2(\%22}{\sage{kzh\underscore{}6\underscore{}slope\underscore{}fulldim\underscore{}covers\underscore{}2}}}
\pgfkeyssetvalue{/sagefunc/kzh_6_slope_fulldim_covers_3}{\href{\githubsearchurl?q=\%22def+kzh_6_slope_fulldim_covers_3(\%22}{\sage{kzh\underscore{}6\underscore{}slope\underscore{}fulldim\underscore{}covers\underscore{}3}}}
\pgfkeyssetvalue{/sagefunc/kzh_6_slope_fulldim_covers_4}{\href{\githubsearchurl?q=\%22def+kzh_6_slope_fulldim_covers_4(\%22}{\sage{kzh\underscore{}6\underscore{}slope\underscore{}fulldim\underscore{}covers\underscore{}4}}}
\pgfkeyssetvalue{/sagefunc/kzh_6_slope_fulldim_covers_5}{\href{\githubsearchurl?q=\%22def+kzh_6_slope_fulldim_covers_5(\%22}{\sage{kzh\underscore{}6\underscore{}slope\underscore{}fulldim\underscore{}covers\underscore{}5}}}
\pgfkeyssetvalue{/sagefunc/kzh_7_slope_1}{\href{\githubsearchurl?q=\%22def+kzh_7_slope_1(\%22}{\sage{kzh\underscore{}7\underscore{}slope\underscore{}1}}}
\pgfkeyssetvalue{/sagefunc/kzh_7_slope_2}{\href{\githubsearchurl?q=\%22def+kzh_7_slope_2(\%22}{\sage{kzh\underscore{}7\underscore{}slope\underscore{}2}}}
\pgfkeyssetvalue{/sagefunc/kzh_7_slope_3}{\href{\githubsearchurl?q=\%22def+kzh_7_slope_3(\%22}{\sage{kzh\underscore{}7\underscore{}slope\underscore{}3}}}
\pgfkeyssetvalue{/sagefunc/kzh_7_slope_4}{\href{\githubsearchurl?q=\%22def+kzh_7_slope_4(\%22}{\sage{kzh\underscore{}7\underscore{}slope\underscore{}4}}}
\pgfkeyssetvalue{/sagefunc/ll_strong_fractional}{\href{\githubsearchurl?q=\%22def+ll_strong_fractional(\%22}{\sage{ll\underscore{}strong\underscore{}fractional}}}
\pgfkeyssetvalue{/sagefunc/mlr_cpl3_a_2_slope}{\href{\githubsearchurl?q=\%22def+mlr_cpl3_a_2_slope(\%22}{\sage{mlr\underscore{}cpl3\underscore{}a\underscore{}2\underscore{}slope}}}
\pgfkeyssetvalue{/sagefunc/mlr_cpl3_b_3_slope}{\href{\githubsearchurl?q=\%22def+mlr_cpl3_b_3_slope(\%22}{\sage{mlr\underscore{}cpl3\underscore{}b\underscore{}3\underscore{}slope}}}
\pgfkeyssetvalue{/sagefunc/mlr_cpl3_c_3_slope}{\href{\githubsearchurl?q=\%22def+mlr_cpl3_c_3_slope(\%22}{\sage{mlr\underscore{}cpl3\underscore{}c\underscore{}3\underscore{}slope}}}
\pgfkeyssetvalue{/sagefunc/mlr_cpl3_d_3_slope}{\href{\githubsearchurl?q=\%22def+mlr_cpl3_d_3_slope(\%22}{\sage{mlr\underscore{}cpl3\underscore{}d\underscore{}3\underscore{}slope}}}
\pgfkeyssetvalue{/sagefunc/mlr_cpl3_f_2_or_3_slope}{\href{\githubsearchurl?q=\%22def+mlr_cpl3_f_2_or_3_slope(\%22}{\sage{mlr\underscore{}cpl3\underscore{}f\underscore{}2\underscore{}or\underscore{}3\underscore{}slope}}}
\pgfkeyssetvalue{/sagefunc/mlr_cpl3_g_3_slope}{\href{\githubsearchurl?q=\%22def+mlr_cpl3_g_3_slope(\%22}{\sage{mlr\underscore{}cpl3\underscore{}g\underscore{}3\underscore{}slope}}}
\pgfkeyssetvalue{/sagefunc/mlr_cpl3_h_2_slope}{\href{\githubsearchurl?q=\%22def+mlr_cpl3_h_2_slope(\%22}{\sage{mlr\underscore{}cpl3\underscore{}h\underscore{}2\underscore{}slope}}}
\pgfkeyssetvalue{/sagefunc/mlr_cpl3_k_2_slope}{\href{\githubsearchurl?q=\%22def+mlr_cpl3_k_2_slope(\%22}{\sage{mlr\underscore{}cpl3\underscore{}k\underscore{}2\underscore{}slope}}}
\pgfkeyssetvalue{/sagefunc/mlr_cpl3_l_2_slope}{\href{\githubsearchurl?q=\%22def+mlr_cpl3_l_2_slope(\%22}{\sage{mlr\underscore{}cpl3\underscore{}l\underscore{}2\underscore{}slope}}}
\pgfkeyssetvalue{/sagefunc/mlr_cpl3_n_3_slope}{\href{\githubsearchurl?q=\%22def+mlr_cpl3_n_3_slope(\%22}{\sage{mlr\underscore{}cpl3\underscore{}n\underscore{}3\underscore{}slope}}}
\pgfkeyssetvalue{/sagefunc/mlr_cpl3_o_2_slope}{\href{\githubsearchurl?q=\%22def+mlr_cpl3_o_2_slope(\%22}{\sage{mlr\underscore{}cpl3\underscore{}o\underscore{}2\underscore{}slope}}}
\pgfkeyssetvalue{/sagefunc/mlr_cpl3_p_2_slope}{\href{\githubsearchurl?q=\%22def+mlr_cpl3_p_2_slope(\%22}{\sage{mlr\underscore{}cpl3\underscore{}p\underscore{}2\underscore{}slope}}}
\pgfkeyssetvalue{/sagefunc/mlr_cpl3_q_2_slope}{\href{\githubsearchurl?q=\%22def+mlr_cpl3_q_2_slope(\%22}{\sage{mlr\underscore{}cpl3\underscore{}q\underscore{}2\underscore{}slope}}}
\pgfkeyssetvalue{/sagefunc/mlr_cpl3_r_2_slope}{\href{\githubsearchurl?q=\%22def+mlr_cpl3_r_2_slope(\%22}{\sage{mlr\underscore{}cpl3\underscore{}r\underscore{}2\underscore{}slope}}}
\pgfkeyssetvalue{/sagefunc/rlm_dpl1_extreme_3a}{\href{\githubsearchurl?q=\%22def+rlm_dpl1_extreme_3a(\%22}{\sage{rlm\underscore{}dpl1\underscore{}extreme\underscore{}3a}}}
\pgfkeyssetvalue{/sagefunc/automorphism}{\href{\githubsearchurl?q=\%22def+automorphism(\%22}{\sage{automorphism}}}
\pgfkeyssetvalue{/sagefunc/multiplicative_homomorphism}{\href{\githubsearchurl?q=\%22def+multiplicative_homomorphism(\%22}{\sage{multiplicative\underscore{}homomorphism}}}
\pgfkeyssetvalue{/sagefunc/projected_sequential_merge}{\href{\githubsearchurl?q=\%22def+projected_sequential_merge(\%22}{\sage{projected\underscore{}sequential\underscore{}merge}}}
\pgfkeyssetvalue{/sagefunc/restrict_to_finite_group}{\href{\githubsearchurl?q=\%22def+restrict_to_finite_group(\%22}{\sage{restrict\underscore{}to\underscore{}finite\underscore{}group}}}
\pgfkeyssetvalue{/sagefunc/interpolate_to_infinite_group}{\href{\githubsearchurl?q=\%22def+interpolate_to_infinite_group(\%22}{\sage{interpolate\underscore{}to\underscore{}infinite\underscore{}group}}}
\pgfkeyssetvalue{/sagefunc/two_slope_fill_in}{\href{\githubsearchurl?q=\%22def+two_slope_fill_in(\%22}{\sage{two\underscore{}slope\underscore{}fill\underscore{}in}}}
\pgfkeyssetvalue{/sagefunc/generate_example_e_for_psi_n}{\href{\githubsearchurl?q=\%22def+generate_example_e_for_psi_n(\%22}{\sage{generate\underscore{}example\underscore{}e\underscore{}for\underscore{}psi\underscore{}n}}}
\pgfkeyssetvalue{/sagefunc/chen_3_slope_not_extreme}{\href{\githubsearchurl?q=\%22def+chen_3_slope_not_extreme(\%22}{\sage{chen\underscore{}3\underscore{}slope\underscore{}not\underscore{}extreme}}}
\pgfkeyssetvalue{/sagefunc/psi_n_in_bccz_counterexample_construction}{\href{\githubsearchurl?q=\%22def+psi_n_in_bccz_counterexample_construction(\%22}{\sage{psi\underscore{}n\underscore{}in\underscore{}bccz\underscore{}counterexample\underscore{}construction}}}
\pgfkeyssetvalue{/sagefunc/gomory_fractional}{\href{\githubsearchurl?q=\%22def+gomory_fractional(\%22}{\sage{gomory\underscore{}fractional}}}
\pgfkeyssetvalue{/sagefunc/not_minimal_2}{\href{\githubsearchurl?q=\%22def+not_minimal_2(\%22}{\sage{not\underscore{}minimal\underscore{}2}}}
\pgfkeyssetvalue{/sagefunc/not_extreme_1}{\href{\githubsearchurl?q=\%22def+not_extreme_1(\%22}{\sage{not\underscore{}extreme\underscore{}1}}}
\pgfkeyssetvalue{/sagefunc/kzh_2q_example_1}{\href{\githubsearchurl?q=\%22def+kzh_2q_example_1(\%22}{\sage{kzh\underscore{}2q\underscore{}example\underscore{}1}}}
\pgfkeyssetvalue{/sagefunc/zhou_two_sided_discontinuous_cannot_assume_any_continuity}{\href{\githubsearchurl?q=\%22def+zhou_two_sided_discontinuous_cannot_assume_any_continuity(\%22}{\sage{zhou\underscore{}two\underscore{}sided\underscore{}discontinuous\underscore{}cannot\underscore{}assume\underscore{}any\underscore{}continuity}}}
\pgfkeyssetvalue{/sagefunc/extremality_test}{\href{\githubsearchurl?q=\%22def+extremality_test(\%22}{\sage{extremality\underscore{}test}}}
\pgfkeyssetvalue{/sagefunc/plot_2d_diagram}{\href{\githubsearchurl?q=\%22def+plot_2d_diagram(\%22}{\sage{plot\underscore{}2d\underscore{}diagram}}}
\pgfkeyssetvalue{/sagefunc/nice_field_values}{\href{\githubsearchurl?q=\%22def+nice_field_values(\%22}{\sage{nice\underscore{}field\underscore{}values}}}
\pgfkeyssetvalue{/sagefunc/ParametricRealFieldElement}{\href{\githubsearchurl?q=\%22def+ParametricRealFieldElement(\%22}{\sage{ParametricRealFieldElement}}}
\pgfkeyssetvalue{/sagefunc/ParametricRealField}{\href{\githubsearchurl?q=\%22def+ParametricRealField(\%22}{\sage{ParametricRealField}}}
\pgfkeyssetvalue{/sagefunc/kzh_minimal_has_only_crazy_perturbation_1}{\href{\githubsearchurl?q=\%22def+kzh_minimal_has_only_crazy_perturbation_1(\%22}{\sage{kzh\underscore{}minimal\underscore{}has\underscore{}only\underscore{}crazy\underscore{}perturbation\underscore{}1}}}
\pgfkeyssetvalue{/sagefunc/bcds_discontinuous_everywhere}{\href{\githubsearchurl?q=\%22def+bcds_discontinuous_everywhere(\%22}{\sage{bcds\underscore{}discontinuous\underscore{}everywhere}}}

\DeclareRobustCommand\sage[1]{\textsf{\upshape #1}}
\DeclareRobustCommand\sagefunc[1]{\pgfkeys{/sagefunc/#1}}

\providecommand\tightslack[1]{#1}   

\newcommand\DIFFPROTECT[1]{#1}

\titlerunning{Facets, weak facets, and extreme functions}
\title{Facets, weak facets, and extreme functions\\
  of the Gomory--Johnson infinite group problem%
  \thanks{The authors gratefully acknowledge partial support from the National Science
  Foundation through grant DMS-1320051, awarded to M.~K\"oppe.
  A preliminary version appeared in Chapter 6 of the second author's Ph.D. thesis
  \emph{Infinite-dimensional relaxations of mixed-integer optimization 
  problems}, University of California, Davis, Graduate Group in
  Applied Mathematics, May 2017, available from
  \url{https://search.proquest.com/docview/1950269648}.
  An extended abstract appeared in:
  M.~K{\"o}ppe and Y.~Zhou, \emph{On the notions of facets, weak facets,
  and extreme functions of the {G}omory--{J}ohnson infinite group problem},
  Integer Programming and Combinatorial Optimization: 19th International
  Conference, IPCO 2017, Waterloo, ON, Canada, June 26--28, 2017, Proceedings
  (Friedrich Eisenbrand and Jochen Koenemann, eds.), Springer International
  Publishing, Cham, 2017, pp.~330--342, \url{https://doi.org/10.1007/978-3-319-59250-3_27}, ISBN
  {978-3-319-59250-3}.
}}

\author{Matthias K\"oppe \and
  Yuan Zhou}

\institute{Dept.\ of Mathematics, University of California, Davis\\
  \texttt{mkoeppe@math.ucdavis.edu} 
  \and
  Dept.\ of Mathematics, University of Kentucky\\
  \texttt{yuan.zhou@uky.edu}}

\date{$\relax$Revision: 3352 $ - \ $Date: 2019-11-14 10:09:31 -0500 (Thu, 14 Nov 2019) $ $\!\!\!}

\usepackage[normalem]{ulem}

\begin{document}
 \newcommand{\tgreen}[1]{\textsf{\textcolor {ForestGreen} {#1}}}
 \newcommand{\tred}[1]{\texttt{\textcolor {red} {#1}}}
 \newcommand{\tblue}[1]{\textcolor {blue} {#1}}

\maketitle

\begin{abstract}
  We investigate three competing notions that generalize the notion of a facet
  of finite-dimensional polyhedra to the infinite-dimension\-al Gomory--Johnson
  model.  These notions were known to coincide for continuous piecewise linear
  functions with rational breakpoints.  We show that two of the notions, extreme functions and
  facets, coincide for the case of continuous piecewise linear functions,
  removing the hypothesis regarding rational breakpoints.
  We prove an if-and-only-if version of the Gomory--Johnson Facet Theorem.
  Finally, we separate the three notions using discontinuous examples.
\end{abstract}

\section{Introduction}

\subsection{Facets in the finite-dimensional case}
Let $G$ be a finite index set.  The space~$\R^{(G)}$ of real-valued functions
$y\colon G\to\R$ is isomorphic to and routinely identified with the Euclidean
space $\R^{|G|}$.  Let $\R^{G}$ denote its dual space.  It is the space of
functions $\alpha\colon G \to \R$, which we consider as linear functionals on
$\R^{(G)}$ via the pairing $\langle \alpha, y \rangle = \sum_{r\in G} \alpha(r)
y(r)$.  Again it is routinely identified with the Euclidean space $\R^{|G|}$, 
and the dual pairing $\langle \alpha, y\rangle $ is the Euclidean inner product.
A~(closed, convex) rational polyhedron 
of~$\R^{(G)}$ is the set of $y\colon
G\to\R$ satisfying $\langle\alpha_i, y\rangle \geq \alpha_{i, 0}$, where
$\alpha_i\in \Z^G$ are integer linear functionals and $\alpha_{i,
  0}\in\Z$, for $i$ ranging over another finite index set~$I$. 
We refer to \cite{sch,ccz-ipbook} for the standard notions of polyhedral geometry.

Consider an integer linear optimization problem in $\R^{(G)}$, i.e.,
the problem of minimizing a linear functional $\eta \in \R^G$ over 
a feasible set $F 
\subseteq \{\, y\colon G \to \Z_+\,\}
\subset \R_+^{(G)}$
, or, equivalently, 
over the convex hull $R = \conv F \subset \R_+^{(G)}$.  
A \emph{valid inequality} for $R$ is an inequality of the form
$\langle \pi, y \rangle \geq \pi_0$, where $\pi \in \R^G$, which holds for
all $y \in R$ (equivalently, for all $y \in F$).  If $R$ is closed, it is exactly the set of all $y$ that satisfy all valid inequalities. 
In the following we will restrict ourselves to the case that $R \subseteq
\R_+^{(G)}$ is a polyhedron of blocking type \cite[section~9.2]{sch}, 
i.e., a polyhedron in $\R_+^{(G)}$ whose recession cone is the positive
orthant.  Then it suffices to consider normalized valid inequalities $\langle
\pi, y \rangle \geq \pi_0$ with $\pi\geq0$ and $\pi_0=1$,
together with the trivial inequalities $y(r) \geq 0$.

Let $P(\pi)$ denote the set of functions $y \in F$ for which the inequality
$\langle \pi, y \rangle \geq 1$ is tight, i.e.,
$\langle \pi, y \rangle = 1$.  If $P(\pi) \neq \emptyset$, then
$\langle \pi, y \rangle \geq 1$ is a \emph{tight valid inequality}.  Then
$R$ is exactly the set of all $y\geq0$ that satisfy all tight valid inequalities.
A valid inequality $\langle \pi, y \rangle \geq 1$ is called
\emph{minimal} 
if there is no other valid inequality $\langle \pi', y \rangle \geq 1$ where $\pi' \neq \pi$ such that $\pi' \leq \pi$ pointwise.
One can show that a minimal valid inequality is tight.
A valid inequality $\langle \pi, y \rangle \geq 1$ is called
\emph{facet-defining} if 
\begin{equation}\tag{wF}
  \begin{aligned}
    \text{for every valid inequality $\langle \pi', y \rangle \geq 1$
      such that $P(\pi)\subseteq P(\pi')$},\\
    \text{we have $P(\pi)=P(\pi')$,}
  \end{aligned}
  \label{eq:facet-definition-like-weak-facet}
\end{equation}
or, in other words, if the face induced by $\langle \pi, y \rangle \geq 1$
is maximal 
\DIFFPROTECT{\cite[section~8.4]{sch}.} 
Because $R$ is of blocking type, 
it has full affine dimension \cite[section~9.2]{sch}.  
\DIFFPROTECT{Hence, there is a unique minimal representation of~$R$ by
  a finite system of linear inequalities
  (up to reordering them and multiplying them by positive real numbers) 
  which are in bijection with the facets \cite[section~8.4]{sch}.
  Because of our normalization, this implies the 
  following two equivalent characterizations of facet-defining inequalities of the form $\langle
\pi', y \rangle \geq 1$:}
\begin{equation}\tag{F}
  \begin{aligned}
    \text{for every valid inequality $\langle \pi', y \rangle \geq 1$
      such that $P(\pi)\subseteq P(\pi')$},\\
    \text{we have $\pi = \pi'$,}
  \end{aligned}
  \label{eq:facet-definition-like-facet}
\end{equation}
and
\begin{equation}\tag{E}
   \begin{aligned}
     \text{if $\langle \pi^1, y \rangle \geq 1$ and $\langle \pi^2, y \rangle
       \geq 1$ are valid inequalities, and $\pi = \tfrac12 (\pi^1+\pi^2)$}\\
     \text{then $\pi = \pi^1 = \pi^2$}. 
     \label{eq:facet-definition-like-extreme-function}
   \end{aligned}
\end{equation}




\subsection{Facets in the infinite-dimensional Gomory--Johnson model}



It is perhaps not surprising that the three conditions
\eqref{eq:facet-definition-like-weak-facet},
\eqref{eq:facet-definition-like-facet}, and
\eqref{eq:facet-definition-like-extreme-function} are no longer equivalent
when $R$ is a general convex set that is not polyhedral, and in particular when we change from the
finite-dimensional to the infinite-dimensional setting.
In the present paper, however, we consider a particular case of an
infinite-dimensional model, in which this question has eluded researchers for
a long time. Let $G=\Q$ or $G=\R$ and let $\R^{(G)}$ now denote the space of
finite-support functions $y\colon G\to\R$. 
The so-called \emph{infinite group problem} was introduced
by Gomory and Johnson in their seminal papers \cite{infinite,infinite2}. 
Let $F = F_{f}(G,\Z)\subseteq \R_+^{(G)}$ be the set of all finite-support
functions $y\colon G\to\Z_+ $ satisfying the equation 
\begin{equation}
  \label{GP} 
  \sum_{r \in G} r\, y(r) \equiv f \pmod{1}
\end{equation}
where $f$ is a given element of $G\setminus \Z$. 
We study its convex hull $R = R_{f}(G,\Z) \subseteq \R_+^{(G)}$, 
consisting of the functions $y\colon G \to \R_+$ 
that can be written as (finite) convex combinations of elements of~$F$, 
and which are therefore finite-support functions as well.

Valid inequalities for $R$ are of the form
$\langle \pi, y \rangle \geq \pi_0$, where $\pi$ comes from the dual space
$\R^G$, which is the space of all real-valued functions (without the
finite-support condition).  When $G=\Q$, then $R$ is again of ``blocking
type'' (see, for example, \cite[section 5]{corner_survey}), and so we again
may assume $\pi\geq0$ and $\pi_0=1$.

If $G=\R$ (the setting of the present paper), typical pathologies from the
analysis of functions of a real 
variable come into play.  By \cite[Proposition 2.4]{igp_survey},
there is an infinite-dimensional subspace $\Pi^* \subset \R^G$ of functions
$\pi^*$ such that the equations $\langle\pi^*,
y\rangle=0$ are valid for~$R$.
The functions~$\pi^*$ are constructed using a Hamel basis of~$\R$
over~$\Q$, and each $\pi^*\in\Pi^*$, $\pi^*\neq0$ has a graph whose topological
closure is~$\R^2$.  Recently, Basu et al.~\cite[Theorem
3.5]{basu2016structure} showed that for every valid inequality $\langle \pi, y
\rangle \geq \pi_0$ there exists
a valid inequality $\langle \pi', y \rangle \geq \pi_0$
with $\pi'\geq0$ such that $\pi' - \pi \in \Pi^*$.  Thus, ignoring trivial
inequalities with $\pi_0 \leq 0$, we may once again assume $\pi\geq0$ and
normalize to $\pi_0=1$.  We call such functions $\pi$ \emph{valid functions}.
In contrast to Gomory and Johnson \cite{infinite,infinite2}, who only considered continuous
functions~$\pi$, this class of functions contains many interesting
discontinuous functions such as
the Gomory fractional cut.

(Minimal) valid functions~$\pi$ that satisfy the conditions
\eqref{eq:facet-definition-like-weak-facet},
\eqref{eq:facet-definition-like-facet}, and
\eqref{eq:facet-definition-like-extreme-function}, are called \emph{weak
  facets}, \emph{facets}, and \emph{extreme functions}, respectively. 
The relation of these notions, in particular of facets and extreme functions, 
has remained unclear in the literature.  For example, Basu et
al. \cite{bccz08222222}, responding to a claim by Gomory and Johnson in
\cite{tspace}, wrote:  
\begin{quote}
  The statement that extreme functions are facets appears to be quite
  nontrivial to prove, and to the best of our knowledge there is no proof in
  the literature. We therefore cautiously treat extreme functions and facets
  as distinct concepts, and leave their equivalence as an open question.
\end{quote}

The survey \cite[section 2.2]{igp_survey} summarizes what was known
about the relation of the three notions:  Facets form a subset of the
intersection of extreme functions and weak facets; see \autoref{fig:separation}.  
For the family $\mathcal F_1$ of continuous piecewise linear functions with 
  rational breakpoints, \cite[Proposition 2.8]{igp_survey} and
  \cite[Theorem~8.6]{igp_survey_part_2} proved that (E)
  $\Leftrightarrow$ (F). Moreover, in this case, (wF)
  $\Rightarrow$ (F) can be shown easily as another consequence of
  \cite[Theorem~8.6]{igp_survey_part_2}
  . 
Thus (E), (F), (wF) are equivalent when $\pi$ is a
continuous piecewise linear function with rational breakpoints. 


\subsection{Contribution of this paper}

A well known sufficient condition for facetness of a minimal valid function~$\pi$ is
the Gomory--Johnson Facet Theorem.  In its strong form, due to
Basu--Hildebrand--K\"{o}ppe--Molinaro \cite{basu-hildebrand-koeppe-molinaro:k+1-slope}, it reads:
\begin{theorem}[{Facet Theorem, strong form, \cite[Lemma
34]{basu-hildebrand-koeppe-molinaro:k+1-slope}; see also \cite[Theorem
2.12]{igp_survey}}]
\label{thm:facet-theorem-strong-form}
  Suppose for every minimal valid function $\pi'$, $E(\pi) \subseteq E(\pi')$
  implies $\pi' = \pi$.  Then $\pi$ is a facet.
\end{theorem}
(Here $E(\pi)$ is the \emph{additivity domain} of $\pi$, defined in
\autoref{sec:notations}.) 
We show (\autoref{lemma:facet_theorem} below) that, in fact, \textbf{this holds as an
``if and only if'' statement.} 
The technique of the proof of this converse is not surprising, but the result
is crucial for the remainder of the paper, and it closes a gap in the literature.

As we mentioned above, for the family $\mathcal F_1$ of continuous piecewise linear functions with rational 
breakpoints, Basu et al. \cite[Proposition 2.8]{igp_survey} showed that the
notions of extreme functions and facets coincide.  This was a consequence of
Basu et al.'s finite oversampling theorem, which connects the extremality of a function $\pi\in\mathcal F_1$ to the extremality of its restriction in a finite group problem \cite{basu-hildebrand-koeppe:equivariant}. 
We \textbf{sharpen this result by removing the hypothesis regarding rational
breakpoints.}

\begin{theorem} 
  \label{lemma:cont_pwl_extreme_is_facet}
  Let $\mathcal F_4$ be the family of continuous piecewise linear functions (not necessarily with rational breakpoints).
  Then 
  $$ \{\, \pi\in \mathcal F_4 : \text{$\pi$ is extreme} \,\}
  =  \{\, \pi\in \mathcal F_4 : \text{$\pi$ is a facet} \,\}.$$
\end{theorem}
The proof relies on our new characterization of facets, as well as on a
technical development on so-called \emph{effective perturbation functions} in
\autoref{s:effective}, which is also of independent interest.

Then we investigate the notions of facets and weak facets in the case of
discontinuous functions.  This appears to be a first in the published
literature.  All papers that consider discontinuous functions only used the
notion of extreme functions.  

We give \textbf{three discontinuous functions that furnish the separation of the three
notions} (\autoref{s:discontinuous_examples}): 
A function $\psi$ that is extreme, but is neither a weak facet nor a facet;
a function $\pi$ that is not an extreme function (nor a facet), but is a weak
facet;
and a function $\pilifted$ that is extreme and a weak facet but is
not a facet; see \autoref{fig:separation}. Two of these three separations are obtained by
extending a rather complicated construction from the authors' paper
\cite{koeppe-zhou:crazy-perturbation}; the proofs are in part
computer-assisted.

\begin{figure}[tp]
  \centering
  \begin{tikzpicture}
    [scale=0.9, 
    font=\small, valid/.style={rounded
      corners,fill=validcolor!20}, minimal/.style={rounded
      corners,fill=minimalcolor!20}, extreme/.style={rounded
      corners,fill=extremecolor!20}, convexcomb/.style={very thick},
    domination/.style={}, facet/.style={rounded corners,fill=facetcolor!20},
    weakfacet/.style={rounded corners,fill=weakfacetcolor!20}] \draw[valid]
    (0,1)--(0,6)--(7,6)--(7,1)--cycle; \draw[minimal]
    (.5,1.5)--(.5,5)--(6.5,5)--(6.5,1.5)--cycle; \draw[extreme] (2.5,3) ellipse
    (1.75 and 1.25); \draw[weakfacet, opacity = .5] (4.5,3) ellipse (1.75 and
    1.25); \draw (2.5,3) ellipse (1.75 and 1.25); \draw (4.5,3) ellipse (1.75
    and 1.25); \draw[facet] (3.5,3) ellipse (.75 and .6); \node at (3.5, 5.5)
    (VALIDTEXT) { valid functions}; \node at (3.5, 4.6) { minimal functions};
    \node at (1.8, 3.25) { extreme}; \node at (1.8, 2.75) { functions}; \node
    at (5.1, 3.25) {weak}; \node at (5.1, 2.75) {facets}; \node at (3.5, 3) {
      facets};
    \node[color=blue] at (2.0, 2.2) {$\psi$}; 
    \node[color=blue] at (4.7, 2.2) {$\pi$};
    \node[color=blue] at (3.5, 2.2) {$\hat\pi$};
  \end{tikzpicture}
  \caption{Separation of the three notions in the discontinuous case}
  \label{fig:separation}
\end{figure}
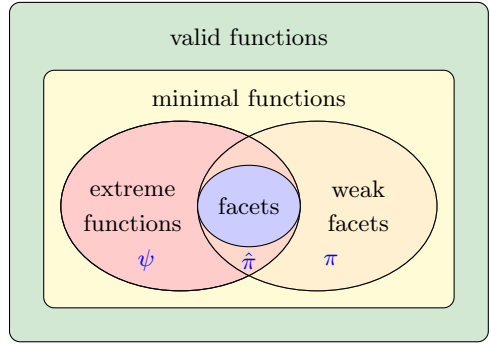

It remains an open question whether this separation can also be done using
continuous (necessarily non--piecewise linear) functions.  We discuss this
question in the conclusions of the paper, \autoref{s:conclusion}.

\section{Minimal valid functions and their perturbations}
\label{sec:notations}
Following \cite{igp_survey}, we define possibly discontinuous piecewise linear
functions~$\pi$ on~$\R$ as follows.  Take a collection $\P_1$ of closed
proper intervals (\emph{one-dimensional faces}) $I \subseteq \R$ such that $\R
= \bigcup \P_1$ (\emph{completeness}) and the intersection of any two distinct $I_1, I_2 \in \P_1$
is either empty or a singleton that consists of a common endpoint of $I_1$
and~$I_2$ (\emph{face-to-face property}).  \DIFFPROTECT{Let $\P_0$ be 
the set of singletons (\emph{zero-dimensional faces, vertices}) arising as
intersections $I_1\cap I_2$ for $I_1, I_2 \in 
\P_1$. Define $\P = \{\emptyset\} \cup \P_0 \cup \P_1$, which we refer to as a
\emph{polyhedral complex}.
\DIFFPROTECT{We assume that it is \emph{locally finite}, i.e., every compact interval of $\R$ has a
nonempty intersection with only finitely many elements of $\P$.}
We call a function~$\pi$ \emph{piecewise linear} over the complex~$\P$
if it is affine linear on the relative interior of each face~$I\in\P$.
This is a nontrivial condition only for the one-dimensional faces $I = [a,b]\in\P_1$, 
for which it means that $\pi$ is affine linear on the open interval $(a,b)$.
To express limits, for $x\in I$ we denote
\begin{equation}
  \label{eq:1d-limit}
  \pi_I(x) = \lim_{\substack{u\to x\\ u \in\relint(I)}} \pi(u).
\end{equation}
We have
\begin{align}
  \pi(x) &= \pi_I(x) \quad\text{for all $x$ in the relative
           interior of the face~$I\in\P$,}\\
  \intertext{%
  and thus $\pi_I$ is the extension of the affine linear function on
  $\relint(F)$ to the closed face~$F$. When $\pi$ is continuous, we have}
  \pi(x) &= \pi_I(x) \quad\text{for all $x$ in the face~$I\in\P$.}
           \label{eq:linear-on-face-continuous-case}
\end{align}}

\begin{example}\label{ex:hildebrand_discont_3_slope_1}
  \begin{figure}[tp]
    \centering
    \includegraphics[width=\linewidth]{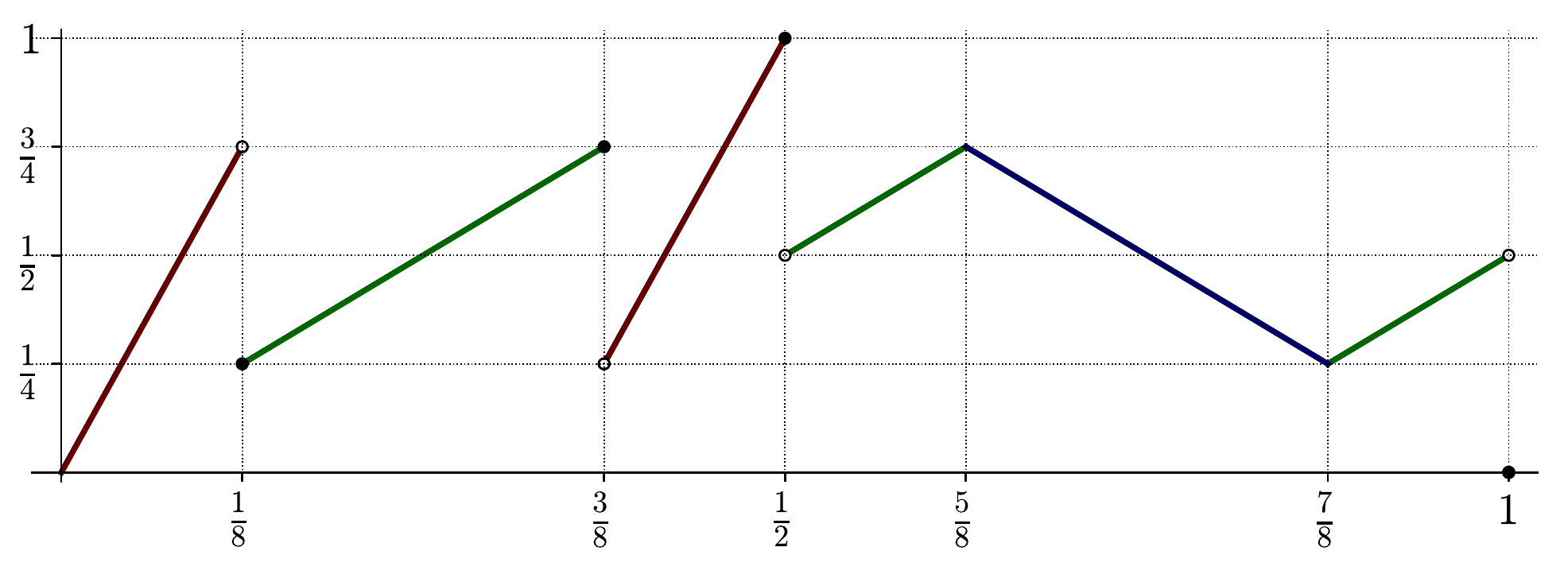}
    \caption{The piecewise linear function
      $\psi = \sage{hildebrand\_discont\_3\_slope\_1()}$.}
    \label{fig:hildebrand_discont_3_slope_1}
  \end{figure}
  Consider the discontinuous piecewise linear function~$\psi$ shown in
  \autoref{fig:hildebrand_discont_3_slope_1}, which will become important in
  \autoref{s:discontinuous_examples}.  It was constructed by Hildebrand (2013,
  unpublished; reported in \cite{igp_survey}) and is available in the
  electronic compendium of extreme functions
  \cite{cutgeneratingfunctionology:lastest-release} as
  \sage{\sagefunc{hildebrand_discont_3_slope_1}()}.
  Here $\P_1$ consists of the one-dimensional faces (closed proper intervals) $[0, \frac18]$,
  $[\frac18,\frac38]$, $[\frac38,\frac12]$, $[\frac12,\frac58]$,
  $[\frac58,\frac78]$, $[\frac78,1]$, and their translations by integers.
  $\P_0$ consists of the singletons corresponding to all endpoints of these
  intervals. 
  For $I=[0, \frac18]\in\P_1$, we have the linear function $\pi_I(x) = 6x$ for
  $x\in I$, 
  and $\pi(x) = \pi_I(x)$ for $x \in \relint(I) = (0,\frac18)$. 
  For $I=\relint(I)=\{\frac18\}\in\P_0$, we have $\pi(\frac18) = \pi_I(\frac18) = \frac14$. 
\end{example}

For a function $\pi\colon \R \to \R$, 
define the \emph{subadditivity slack} of $\pi$ as $\Delta\pi(x,y) := \pi(x) + \pi(y)
- \pi(x+y)$; then $\pi$ is subadditive if and only if $\Delta\pi(x,y)\geq0$
for all $x, y\in \R$. Denote the \emph{additivity domain} of $\pi$ by
$$E(\pi) = \{\,(x,y) \st \Delta\pi(x,y) = 0\,\}.$$
By a theorem of Gomory and
Johnson~\cite{infinite} (see \cite[Theorem 2.6]{igp_survey}), 
the minimal valid functions are
exactly the subadditive functions $\pi\colon \R\to\R_+$
that satisfy $\pi(0) = 0$,
are periodic modulo~$1$ and satisfy the \emph{symmetry condition}
$\pi(x) + \pi(f - x) = 1$ for all $x\in\R$.  As a consequence, minimal valid
functions are bounded between $0$ and~$1$.

To combinatorialize the additivity domains of piecewise linear subadditive
functions, we work with a two-dimensional polyhedral complex $\Delta \P
$.
It is defined as the collection of (closed)
polyhedra 
$$F(I, J, K) = \bigl\{\,(x,y) \in \R \times \R \bigst x \in I,\, y \in J,\,
x+y \in K\,\bigr\}$$
for $I, J, K \in \P$, which we refer to as the \emph{faces} of $\Delta\P$. 
As $I$, $J$, and $K$ can be proper intervals or singletons of $\P$, the
nonempty faces $F$ of $\Delta\P$ can be zero-, one-, or two-dimensional.
\autoref{fig:simple_E_pi_extreme_not_facet} (left) shows $\Delta\P$
corresponding to the function~$\psi$ of \autoref{ex:hildebrand_discont_3_slope_1}.
Define the projections $p_1,p_2,p_3\colon \R\times \R \to \R$ as
$p_1(x,y) = x$,  $p_2(x,y) = y$, $p_3(x,y) = x+y$.

In the continuous case, since the function $\pi$ is piecewise linear over
$\P$, we have by~\eqref{eq:linear-on-face-continuous-case} that $\Delta \pi$
is affine linear over each face $F \in \Delta \P$.  
Let $\pi$ be a  minimal valid function for $R_f(\R, \Z)$ that is piecewise linear over $\P$. 
Following \cite{igp_survey}, we define the \emph{space of perturbation functions with prescribed additivities} $E = E(\pi)$ 
\begin{equation}
\bar{\Pi}^{E}(\R,\Z) = \left\{\bar{\pi} \colon \R \to \R \, \Bigg| \,
\begin{array}{r@{\;}c@{\;}ll}
\bar{\pi}(0) &=& 0 \\
\bar{\pi}(f) &=& 0 \\
\Delta\bar{\pi}(x,y) &=& 0 & \text{ for all } (x,y) \in E\\
\bar{\pi}(x+z) &=& \bar{\pi}(x) & \text{ for all } x \in \R,\, z \in \Z
\end{array} \right\}.
\label{eq:perturbation_space_simple}
\end{equation}

When $\pi$ is discontinuous, one also needs to consider the limit points where
the subadditivity slacks are approaching zero at the relative boundary of a
face. Let $F$ be a face of $\Delta \P $. For $(x,y)\in F$, we denote
\begin{equation}
  \label{eq:delta-pi-f}
  \Delta\pi_F(x,y) := 
  \lim_{\substack{(u,v) \to (x,y)\\ (u,v) \in \relint(F)}} \Delta\pi(u,v).
\end{equation}
(For $(x, y)\in\relint(F)$, we have $\Delta\pi_F(x,y) = \Delta\pi(x,y)$. In
particular, for zero-dimensional faces $F = \{(x, y)\}$, we have $\relint(F) =  \{(x,
y)\}$, so the only sequence considered in the limit is the constant sequence $(x,y)$, 
and thus the limit is just the value $\Delta\pi(x,y)$.)
Define 
\[E_F(\pi) = \{\,(x,y)\in F \st \Delta\pi_F(x,y) \text{ exists, and } \Delta\pi_F(x,y) = 0\,\}.\]
Notice that in the above definition of $E_F(\pi)$, we include the condition that
the limit denoted by $\Delta\pi_F(x,y)$ exists, so that this definition can as well be applied to
functions $\pi$ (and $\bar\pi$) that are not piecewise linear over $\P$. 

We denote by $E_{\bullet}(\pi, \P)$ the family of sets $E_F(\pi)$, indexed by $F\in \Delta\P$.
Define the \emph{space of perturbation functions with prescribed additivities and limit-additivities} $E_{\bullet} = E_{\bullet}(\pi, \P)$
\begin{equation}
\bar{\Pi}^{E_{\bullet}}(\R,\Z) = \left\{\bar{\pi} \colon \R \to \R \, \Bigg| \,
\begin{array}{r@{\;}c@{\;}ll}
\bar{\pi}(0) &=& 0 \\
\bar{\pi}(f) &=& 0 \\
\Delta\bar{\pi}_F(x, y) &=& 0 & \text{ for } (x,y) \in E_F, \ F \in \Delta\P\\
\bar{\pi}(x+z) &=& \bar{\pi}(x) & \text{ for } x \in \R,\, z \in \Z
\end{array} \right\}.
\label{eq:perturbation_space}
\end{equation}
\begin{remark}\label{rem:e-bullet-continuous}
Let $\bar{\pi} \in \bar{\Pi}^E(\R,\Z)$. The third condition of \eqref{eq:perturbation_space_simple} is equivalent to $E(\pi) \subseteq E(\bar{\pi})$. 
Let $\bar{\pi} \in \bar{\Pi}^{E_{\bullet}}(\R,\Z)$. The third condition of \eqref{eq:perturbation_space} is equivalent to $E_F(\pi) \subseteq E_F(\bar{\pi})$ for all faces $F\in \Delta\P$, which is stronger than $E(\pi) \subseteq E(\bar{\pi})$ in \eqref{eq:perturbation_space_simple}.
Thus, in general, $\bar{\Pi}^{E_{\bullet}}(\R,\Z) \subseteq \bar{\Pi}^E(\R,\Z)$.  
If $\pi$ is continuous, then $\Delta\pi_F(x,y) = \Delta\pi(x,y)$ for $(x, y)
\in F$. Therefore, $E(\pi) \subseteq E(\bar{\pi})$ implies that $E_F(\pi) \subseteq E_F(\bar{\pi})$ for all faces $F\in \Delta\P$, hence $\bar{\Pi}^{E_{\bullet}}(\R,\Z) = \bar{\Pi}^E(\R,\Z)$.  
\end{remark}

\section{Effective perturbation functions}
\label{s:effective}

Following \cite{koeppe-zhou:crazy-perturbation}, 
we define the vector space
\begin{equation}
\tilde{\Pi}^{\pi}(\R,\Z) = \left\{\,\tilde{\pi} \colon \R \to \R \, \mid \, \exists \, \epsilon>0 \text{ s.t.\ } \pi^{\pm} = \pi \pm \epsilon\tilde{\pi} \text{ are minimal valid}\,\right\},
\label{eq:effective-perturbation-space}
\end{equation}
whose elements are called \emph{effective perturbation functions} for~$\pi$.
Because of \cite[Lemma 2.11\,(i)]{igp_survey}, a function $\pi$ is
extreme if and only if $\tilde{\Pi}^{\pi}(\R,\Z) = \left\{ 0\right\}$. Note
that every
function $\tilde{\pi} \in \tilde{\Pi}^{\pi}(\R,\Z)$ is bounded
.

It is clear that
if $\tilde\pi \in \tilde{\Pi}^{\pi}(\R,\Z)$, then $\tilde\pi \in
\bar{\Pi}^{E_{\bullet}}(\R,\Z)$, where $E_{\bullet} = E_{\bullet}(\pi, \P)$; see
\cite[Lemma 2.7]{basu-hildebrand-koeppe:equivariant} or
\cite[{Lemma~\ref{crazy:lemma:tight-implies-tight}}]{koeppe-zhou:crazy-perturbation}.

The other direction does not hold in general, but requires additional
hypotheses.  Let $\bar\pi \in \bar{\Pi}^{E_{\bullet}}(\R,\Z)$.  In
\cite[Theorem 3.13]{bhk-IPCOext} (see also \cite[Theorem 3.13]{igp_survey}),
it is proved that if $\pi$ and $\bar\pi$ are continuous
and $\bar\pi$ is piecewise linear, we have
$\bar\pi \in \tilde{\Pi}^{\pi}(\R,\Z)$.  (Similar arguments also appeared in
the earlier literature, for example in the proof of \cite[Theorem
3.2]{basu-hildebrand-koeppe:equivariant}.)

We will need a more general version of this result. 
Consider the following definition.
Given a locally finite complete polyhedral complex $\mathcal{P}$ of $\R$, 
we call a function $\bar\pi\colon \R \to \R$ \textit{piecewise Lipschitz
  continuous} over~$\mathcal{P}$, if it is Lipschitz continuous over the
relative interior of each face of the complex. Under this definition,
piecewise Lipschitz continuous functions can be discontinuous at the relative
boundaries of the faces.
\begin{theorem}
\label{lemma:lipschitz-equiv-perturbation}
Let $\pi\colon\R\to\R$ be a minimal valid function that is piecewise linear over a
locally finite polyhedral complex~$\P$. Let $\bar\pi \in \bar\Pi^{E_{\bullet}}(\R,\Z)$ be a
perturbation function, where $E_{\bullet} = E_{\bullet}(\pi, \P)$. Suppose
that $\bar\pi$ is piecewise Lipschitz continuous over $\P$. Then $\bar\pi$ is
an effective perturbation function, $\bar\pi \in
\tilde \Pi^{\pi}(\R,\Z)$.  
\end{theorem} 


\begin{proof}
Let 
\begin{align*}
m := \min\{\,\Delta\pi_F(x,y) \mid {} & (x,y)\in \verts(\Delta \mathcal{P}),\, F
  \text{ is a face of } \Delta \mathcal{P} \\ & \text { such that } (x,y)\in F \text{ and } \Delta\pi_F(x,y)\neq 0\,\};
\end{align*}
\DIFFPROTECT{Because $\pi$ is minimal, it is periodic modulo~$1$; thus
$\bar\pi\in\bar\Pi^{E_{\bullet}}(\R,\Z)$ implies that $\bar\pi$ is also
periodic modulo~$1$.  Because $\P$ is locally finite, only finitely many faces
of it have a nonempty intersection with $[0, 1]$.  Take a positive number $C$
that is greater than the Lipschitz constant of $\bar\pi$ on the relative
interior of each of these finitely many faces.  Then because of periodicity,
$C$ is larger than the Lipschitz constant on all faces of~$\P$.
Moreover, because $\pi$ is piecewise linear over~$\P$, periodic, and nonconstant (as
$\pi(0) = 0$ and $\pi(f)=1$), all faces 
of~$\P$ are bounded.}  Hence $\bar\pi$ is bounded,
and therefore
\[M := \sup_{(x,y)\in \R^2} \left|\Delta\bar \pi(x,y)\right|\]
is finite.
If $M=0$, then $\pi$ is additive; because it is also piecewise Lipschitz continuous and periodic, it follows that $\bar \pi \equiv 0$, and thus $\bar \pi \in \tilde\Pi^{\pi}(\R,\Z)$ holds trivially. In the following, we assume $M>0$. 
Define 
\(\epsilon := \min\big\{\frac{m}{M}, \frac{m}{8C}\big\}.\)
We also have $m > 0$, since $\pi$ is subadditive and $\Delta\pi$ is non-zero somewhere. Thus, $\epsilon>0$.
Let $\pi^+ = \pi+\epsilon\bar\pi$ and $\pi^- = \pi-\epsilon\bar\pi$, which we collectively refer to
as~$\pi^\pm$. We want
to show that $\pi^{\pm}$ are minimal valid
. 

We claim that $\pi^+$ and $\pi^-$ are subadditive functions. Let $(x, y) \in 
[0,1]^2$. 
Let $F$ be a face of $\Delta \P$ such that $(x, y) \in F$. 
We denote the limit~\eqref{eq:delta-pi-f} of $\pi^\pm$ by
$\Delta\pi_F^{\pm}(x, y)$; 
we will show that it is nonnegative.
First, assume $\Delta\pi_F(x, y)=0$. It follows from $E_F(\pi)\subseteq E_F(\bar\pi)$ that $\Delta\bar\pi_F(x, y)=0$. Therefore, $\Delta\pi_F^{\pm}(x, y) = 0$.
Next, assume $\Delta\pi_F(x, y) \neq 0$. Consider $S=\{\,(u,v) \in F \mid \Delta\pi_F(u,v)=0\,\}$, which is a closed set since $\Delta\pi_F$ is continuous over the face $F$. 

If $S=\emptyset$, then $\Delta\pi_F(u,v) \geq m$ for any $(u,v) \in \verts(F)$. We have $\Delta\pi_F(x,y)\geq m$ by the fact that $\Delta\pi_F$ is affine over $F$. 
Hence, in this case, 
\DIFFPROTECT{\begin{align*}
  \Delta\pi_F^{\pm}(x, y) 
  &= \Delta\pi_F(x, y)  \pm \epsilon\Delta\bar\pi_F(x,y) \\
  & \geq \Delta\pi_F(x, y) -\epsilon\left| \Delta\bar\pi_F(x,y)\right| \geq m - \frac{m}{M}M \geq 0.
\end{align*}}

Now consider the case $S\neq \emptyset$. 
Let $d$ denote the Euclidean distance from $(x,y)$ to $S$. Since $S$ is a
closed set, there exists a point $(x',y') \in S$ such that $(x-x')^2+(y-y')^2 = d^2$. 
Let $I = p_1(F), J = p_2(F)$ and $K = p_3(F)$. Then $x, x' \in I, \; y, y' \in J$ and $x+y, x'+y' \in K$.
It follows from $E_F(\pi)\subseteq E_F(\bar\pi)$ and $\Delta\pi_F(x', y')=0$ that $\Delta\bar\pi_F(x', y')=0$.
Therefore,
\begin{align*}
  \Delta\bar\pi_F(x,y) &= \Delta\bar\pi_F(x,y) -\Delta\bar\pi_F(x',y') \\
                       &= \bar\pi_I(x) - \bar\pi_I(x')+\bar\pi_J(y) - \bar\pi_J(y')+\bar\pi_K(x+y)-\bar\pi_K(x'+y'),
\end{align*}
where $\bar\pi_I(x) = \lim_{u\to x, u\in\relint(I)} \bar\pi(u)$ as in~\eqref{eq:1d-limit}.
Since $\bar\pi$ is Lipschitz continuous over $\relint(I), \relint(J)$ and $\relint(K)$, we have that
\begin{align*}
\left|\bar\pi_I(x) - \bar\pi_I(x')\right| &\leq  C \left| x -x' \right| \leq Cd; \\
\left|\bar\pi_J(y) - \bar\pi_J(y')\right| &\leq  C \left| y -y' \right| \leq Cd; \\
\left|\bar\pi_K(x+y) - \bar\pi_K(x'+y')\right| &\leq  C \left| x+y -x'-y' \right| \leq 2Cd.
\end{align*}
Hence $\left|\Delta\bar\pi_F(x,y)\right| \leq 4Cd$.
Applying a geometric estimate
(\autoref{lemma:affine-function-min-value} in Appendix~\ref{appendix:omitted} with $g=\Delta\pi_F$) shows that
$\Delta\pi_F(x,y)\geq \frac{md}{2}$. 
Therefore, in the case where $S \neq \emptyset$,
\begin{align*}
\Delta\pi_F^{\pm}(x, y) &= \Delta\pi_F(x, y)  \pm \epsilon\Delta\bar\pi_F(x,y) \\
& \geq \Delta\pi_F(x, y) -\epsilon\left| \Delta\bar\pi_F(x,y)\right| 
\geq \frac{md}{2} - \frac{m}{8C}(4Cd) = 0.
\end{align*}

We showed that $\pi^{\pm}$ are subadditive. Since $\bar\pi \in \bar\Pi^E(\R,\Z)$, we have $\pi^{\pm}(0)=\pi(0)=0$ and $\pi^{\pm}(f)=\pi(f) =1$. The last result along with $E(\pi)\subseteq E(\bar\pi)$ imply that $\pi^+(x)+ \pi^+(y)=\pi^-(x)+ \pi^-(y) = 1$ if $x + y \equiv f \pmod 1$. The functions $\pi^\pm$ are non-negative. Indeed, suppose that $\pi^+(x)<0$ for some $x\in\R$, then it follows from the subadditivity that $\pi^+(nx)\leq n\pi^+(x)$ for any $n \in \Z_+$, which is a contradiction to the boundedness of $\pi^+$.

Thus, $\pi^\pm$ are minimal valid functions. We conclude that $\bar \pi \in \tilde\Pi^{\pi}(\R,\Z)$.
\end{proof}

\section{Extreme functions and facets}
In this section, we discuss the relations between the notions of extreme
functions and facets.
We first review the definition of a facet, following
\cite[section 2.2.3]{igp_survey}; cf.\ ibid.\ for a discussion of this notion
in the earlier literature, in particular 
\cite{tspace} 
and \cite{dey3}
.

Let $P(\pi)$ denote the set of functions $y\colon \R \to \Z_+$ with finite
support satisfying \[ \DIFFPROTECT{\sum_{r\in\R}r\, y(r) \equiv f \pmod{1}} \quad \text{ and } \quad
  \sum_{r\in\R}\pi(r)y(r)=1.\] A valid function $\pi$ is called a \emph{facet}
if for every valid function $\pi'$ such that $P(\pi) \subseteq P(\pi')$ we
have that $\pi' =\pi$
. Equivalently, a valid function $\pi$ is a facet if this condition holds
for all such \emph{minimal} valid functions $\pi'$ \cite{basu-hildebrand-koeppe-molinaro:k+1-slope}. 

\begin{remark}
  \label{r:discont-lit}
  In our paper we investigate the notions of facets (and weak facets) in
  particular for the case of discontinuous functions.  This appears to be a
  first in the 
  published literature.  All papers that consider discontinuous functions only
  used the notion of extreme functions.  In particular,
  Dey--Richard--Li--Miller \cite{dey1}, who were the first to consider
  previously known discontinuous functions as first-class members of the
  Gomory--Johnson hierarchy of valid functions, use extreme functions
  exclusively; whereas \cite{dey3}, which was completed by a subset of the
  authors in the same year, uses (weak) facets exclusively.  The same is true
  in Dey's Ph.D. thesis \cite{Dey-thesis}: The notion of extreme functions is
  used in chapters regarding discontinuous functions; whereas the notion of
  facets is used when talking about (2-row) continuous functions.  Dey (2016,
  personal communication) remembers that at that time, he and his coauthors
  were aware that facets were the strongest notion and they would strive to
  establish facetness of valid functions whenever possible.  However, in the
  excellent survey \cite{Richard-Dey-2010:50-year-survey}, facets are no
  longer mentioned and the exposition is in terms of extreme functions.
\end{remark}

\begin{remark}
In the discontinuous case, the additivity in the limit plays a role in
extreme functions, which are characterized by the non-existence of an
effective perturbation function $\tilde{\pi}\not\equiv 0$. However facets (and
weak facets, see the next section) are defined through $P(\pi)$, which does not capture the limiting
additive behavior of $\pi$. The additivity domain $E(\pi)$, which appears in
the Facet Theorem as discussed below, also does not account for additivity in
the limit.
\end{remark}


A well known sufficient condition for facetness of a minimal valid function~$\pi$ is
the Gomory--Johnson Facet Theorem.  We have stated its strong form, due to
Basu--Hildebrand--K\"{o}ppe--Molinaro
\cite{basu-hildebrand-koeppe-molinaro:k+1-slope}, in the introduction as
\autoref{thm:facet-theorem-strong-form}. 
In order to prove our ``if and only if'' version, we need the following
lemma. 
\begin{lemma}
\label{lemma:P_pi_and_E_pi}
Let $\pi$ and $\pi'$ be minimal valid functions. Then
$E(\pi) \subseteq E(\pi')$ if and only if $P(\pi) \subseteq P(\pi')$. 
\end{lemma}
\begin{proof}
The ``if'' direction is proven in
\cite[Theorem 20]{basu-hildebrand-koeppe-molinaro:k+1-slope}; see also \cite[Theorem 2.12]{igp_survey}. 
We now show the ``only if" direction, using the subadditivity of $\pi$. Assume that $E(\pi) \subseteq E(\pi')$. Let $y \in P(\pi)$. Let $\{r_1, r_2, \dots, r_n\}$ denote the finite support of $y$. By definition, the function $y$ satisfies that $y(r_i) \in \Z_+$,  $\sum_{i=1}^n r_i y(r_i) \equiv f \pmod 1$, and $\sum_{i=1}^n \pi(r_i) y(r_i) = 1$. 
Since $\pi$ is a minimal valid function, we have that 
\(1 = \sum_{i=1}^n \pi(r_i) y(r_i) \geq \pi\bigl(\sum_{i=1}^n r_i y(r_i)\bigr) = \pi(f) = 1.\) 
Thus, each subadditivity inequality here is tight for $\pi$, and is also tight for $\pi'$ since $E(\pi) \subseteq E(\pi')$.  We obtain \(\sum_{i=1}^n \pi'(r_i) y(r_i) = \pi'\bigl(\sum_{i=1}^n r_i y(r_i)\bigr) = \pi'(f) = 1,\) which implies that $y \in P(\pi')$. Therefore, $P(\pi) \subseteq P(\pi')$.
\end{proof}

\begin{theorem}[Facet Theorem, ``if and only if'' version]
\label{lemma:facet_theorem}
A minimal valid function $\pi$ is a facet if and only if for every minimal valid function $\pi'$, $E(\pi) \subseteq E(\pi')$ implies $\pi'=\pi$. 
\end{theorem}
\begin{proof}
It follows from the Facet Theorem in the strong form
(\autoref{thm:facet-theorem-strong-form}) and \autoref{lemma:P_pi_and_E_pi}.
\end{proof}

Recall the space $\Pi^{E}(\R,\Z)$ of perturbation functions with prescribed
additivities $E = E(\pi)$ from \autoref{sec:notations}.
In \cite[page 25, section 3.6]{igp_survey}, the Facet Theorem is reformulated in terms of
perturbation functions as follows:
\begin{quote}
  If $\pi$ is not a facet, then 
  there exists a non-zero $\bar \pi \in \bar \Pi^{E(\pi)}(\R,\Z)$ such that $\pi' = \pi + \bar \pi$ is a minimal valid function.   
\end{quote}
The authors of \cite{igp_survey} caution
that this last statement is not an ``if and only if'' statement.
We now prove that actually the following ``if and only if'' version holds.

\begin{lemma}
\label{lemma:facet_no_bar_pi}
A minimal valid function $\pi$ is a facet if and only if there is no non-zero $\bar{\pi} \in \bar\Pi^E(\R,\Z)$, where $E=E(\pi)$, such that $\pi+\bar\pi$ is minimal valid.
\end{lemma}
\begin{proof}
Let $\pi$ be a minimal valid function. 

Assume that $\pi$ is a facet. Let $\bar{\pi} \in \bar\Pi^E(\R,\Z)$ where $E=E(\pi)$ such that $\pi' = \pi+\bar\pi$ is minimal valid. It is clear that $E(\pi) \subseteq E(\pi')$. By \autoref{lemma:facet_theorem}, $\pi' = \pi$. Thus, $\bar\pi \equiv 0$.

Assume there is no non-zero $\bar{\pi} \in \bar\Pi^E(\R,\Z)$, where $E=E(\pi)$, such that $\pi+\bar\pi$ is minimal valid. Let $\pi'$ be a minimal valid function such that $E(\pi) \subseteq E(\pi')$. Consider $\bar\pi = \pi' -\pi$. We have that $\bar{\pi} \in \bar\Pi^E(\R,\Z)$ and that $\pi+\bar\pi = \pi'$ is minimal valid. Then $\bar\pi \equiv 0$ by the assumption. Hence, $\pi' = \pi$. It follows from \autoref{lemma:facet_theorem} that $\pi$ is a facet.
\end{proof}

We will not use this lemma in the following.\medbreak

Now we come to the proof of a main theorem stated in the introduction.

\begin{proof}[of \autoref{lemma:cont_pwl_extreme_is_facet}]
Let $\pi$ be a continuous piecewise linear minimal valid function. As
mentioned in \cite[section 2.2.4]{igp_survey}, \cite[Lemma
1.3]{basu-hildebrand-koeppe-molinaro:k+1-slope}
showed that if $\pi$ is a facet, then $\pi$ is extreme.

We now prove the other direction by contradiction. Suppose that $\pi$ is extreme, but is not a facet. 
Then by \autoref{lemma:facet_theorem}, there exists a minimal valid function $\pi' \neq \pi$ such that $E(\pi)\subseteq E(\pi')$. 
Since $\pi$ is continuous piecewise linear and $\pi(0)=\pi(1)=0$, 
there exists $\delta >0$ such that $\Delta\pi(x,y) =0$ and $\Delta\pi(-x, -y) =0$ for $0 \leq x, y \leq \delta$. The condition $E(\pi) \subseteq E(\pi')$ implies that $\Delta\pi'(x,y) =0$ and $\Delta\pi'(-x, -y) =0$ for $0 \leq x, y \leq \delta$ as well. As the function $\pi'$ is bounded, it follows from the Interval Lemma (see \cite[Lemma 4.1]{igp_survey}, for example) that $\pi'$ is affine linear on $[0, \delta]$ and on $[-\delta, 0]$. We also know that $\pi'(0)=0$ as $\pi'$ is minimal valid. Using the subadditivity, we obtain that $\pi'$ is Lipschitz continuous. 

Let $\bar\pi = \pi' -\pi$. Then $\bar\pi \not\equiv 0$, $\bar\pi \in
\bar\Pi^E(\R, \Z)$ where $E = E(\pi)$, and $\bar\pi$ is Lipschitz
continuous. Since $\pi$ is continuous, we have $\bar\Pi^E(\R, \Z) =
\bar\Pi^{E_{\bullet}}(\R,\Z)$ by \autoref{rem:e-bullet-continuous}. By \autoref{lemma:lipschitz-equiv-perturbation}, there exists $\epsilon>0$ such that $\pi^\pm = \pi \pm \epsilon\bar\pi$ are distinct minimal valid functions. This contradicts the assumption that $\pi$ is an extreme function. 

Thus, the equality $\{\text{extreme functions in $\mathcal F_4$}\} = \{\text{facets in $\mathcal F_4$}\}$ is proved.
\end{proof}



\section{Weak facets}

We first review the definition of a weak facet, following
\cite[section 2.2.3]{igp_survey}; cf.\ ibid.\ for a discussion of this notion
in the earlier literature, in particular 
\cite{tspace} 
and \cite{dey3}
.
A valid function $\pi$ is called a \emph{weak facet} if for every valid
function $\pi'$ such that $P(\pi)\subseteq P(\pi')$ we have $P(\pi)=P(\pi')$.

As we mentioned above, to prove that $\pi$ is a
facet, it suffices to consider $\pi'$ that is minimal valid. The following
lemma shows it is also the case for weak facets. 

\begin{lemma}
\label{lemma:weak_facet_minimal}
\begin{enumerate}[\rm(1)]
\item Let $\pi$ be a valid function. If $\pi$ is a weak facet, then $\pi$ is minimal valid.
\item Let $\pi$ be a minimal valid function. Suppose that for every minimal valid function $\pi'$, we have that $P(\pi)\subseteq P(\pi')$ implies $P(\pi)=P(\pi')$. Then $\pi$ is a weak facet.
\item A minimal valid function $\pi$ is a weak facet if and only if for every
  minimal valid function $\pi'$, we have that $E(\pi)\subseteq E(\pi')$
  implies $E(\pi)=E(\pi')$.
\end{enumerate}
\end{lemma}
\begin{proof}
(1) Suppose that $\pi$ is not minimal valid. Then, by \cite[Theorem
1]{basu-hildebrand-koeppe-molinaro:k+1-slope}, $\pi$ is dominated by another
minimal valid function $\pi'$, with $\pi(x_0) > \pi'(x_0)$ at some $x_0$. 
Let $y \in P(\pi)$. We have \DIFFPROTECT{\[1 = \sum_{r\in\R}\pi(r)y(r) \geq \sum_{r\in\R}\pi'(r)y(r) \geq \pi'\bigl(\sum_{r\in\R} r\, y(r)\bigr) = \pi'(f) =1.\]}
Hence equality holds throughout, implying that $y \in P(\pi')$. Therefore, $P(\pi)\subseteq P(\pi')$. Now consider $y$ with $y(x_0)=y(f-x_0)=1$ and $y(x)=0$ otherwise. It is easy to see that $y \in P(\pi')$, but $y \not\in P(\pi)$ since $\pi(x_0)+\pi(f-x_0) > \pi'(x_0)+\pi'(f-x_0) = 1$. Therefore, $P(\pi)\subsetneq P(\pi')$, a contradiction to the weak facet assumption on $\pi$.\smallbreak

(2) Consider any valid function $\pi^*$ (not necessarily minimal) such that
$P(\pi) \subseteq P(\pi^*)$. Let $\pi'$ be a minimal function that dominates
$\pi^*$: $\pi' \leq \pi^*$. From the proof of~(1) we know that $P(\pi^*) \subseteq P(\pi')$. Thus, $P(\pi)\subseteq P(\pi')$. By hypothesis,  we have that $P(\pi)=P(\pi^*)=P(\pi')$. Therefore, $\pi$ is a weak facet.\smallskip

(3) Direct consequence of (2) and \autoref{lemma:P_pi_and_E_pi}.
\end{proof}

By \autoref{lemma:cont_pwl_extreme_is_facet}, for continuous piecewise linear
functions, the notions of extreme functions and facets are the same.  Next we
discuss the relation to weak facets.  We have the following theorem.


\begin{theorem}\label{thm:implications-of-pwl-effective-perturbation}
Let $\mathcal{F}$ be a subfamily of the family $\mathcal F_4$ of continuous
piecewise linear functions 
such that
\begin{equation*}
  \begin{aligned}
    &\text{existence of an effective perturbation for any minimal valid $\pi
    \in \mathcal{F}$} \\
    &\text{implies existence of a piecewise linear effective perturbation}.
  \end{aligned}
\end{equation*}
Let $\pi \in \mathcal{F}$. The following are equivalent.
\textup{(E)} $\pi$ is extreme, \textup{(F)} $\pi$ is a facet, \textup{(wF)} $\pi$ is a weak facet.
\end{theorem}

Before proving the theorem, we discuss a hierarchy of known subfamilies that
satisfy the hypothesis.

\begin{remark}
  As shown in
  \cite{basu-hildebrand-koeppe:equivariant} (for a stronger statement, see
  \cite[Theorem~8.6]{igp_survey_part_2}), the family $\mathcal F_1$ of continuous piecewise
  linear functions with rational breakpoints is such a subfamily 
  where existence of an effective perturbation implies existence of a
  piecewise linear effective perturbation.
\end{remark}

\begin{remark}
  Zhou \cite[Chapter 4]{zhou:dissertation} introduces a completion procedure
  for deciding the extremality of piecewise linear functions, which is known
  to terminate for all functions with rational breakpoints and some functions
  with irrational breakpoints.  Let $\mathcal F_2 \supset \mathcal F_1$ be
  the family of continuous piecewise linear functions with rational
  breakpoints for which the procedure terminates. 
  In this case, by \cite[Lemma 4.11.3, Theorems 4.11.4,
  4.11.6]{zhou:dissertation}, the space of effective perturbations has a 
  precise description as a direct sum of a finite-dimensional space of 
  continuous piecewise linear functions and finitely many spaces of Lipschitz functions.  Because 
  the spaces of Lipschitz functions contain nonzero piecewise linear functions, 
  this implies that $\mathcal F_2$ satisfies the hypothesis of
  \autoref{thm:implications-of-pwl-effective-perturbation}.
\end{remark}

\begin{remark}
  Hildebrand--K\"oppe--Zhou \cite{hildebrand-koeppe-zhou:algo-paper-abstract-ipco,koeppe-zhou:algo-paper}
  consider the family $\mathcal F_3 \supseteq \mathcal F_2$ of
  continuous piecewise linear functions that have a \emph{finitely presented moves
    closure} \cite[Assumption 4.2]{hildebrand-koeppe-zhou:algo-paper-abstract-ipco}.  For these functions, by
  \cite[Theorems 4.14--4.16]{hildebrand-koeppe-zhou:algo-paper-abstract-ipco},
  the space of effective perturbations has a direct sum decomposition of the same type as
  for the family $\mathcal F_2$, and again this implies that the family
  $\mathcal F_3$ satisfies the hypothesis of 
  \autoref{thm:implications-of-pwl-effective-perturbation}.
\end{remark}
\begin{openquestion}
  It is an open question whether the whole family $\mathcal F_4 \supseteq
  \mathcal F_3$ of all continuous piecewise linear 
  functions satisfies the hypothesis of
  \autoref{thm:implications-of-pwl-effective-perturbation}.   
\end{openquestion}
\begin{proof}[of \autoref{thm:implications-of-pwl-effective-perturbation}]
By \autoref{lemma:cont_pwl_extreme_is_facet} and the fact that $\{$facets$\} \subseteq \{$extreme functions$\} \cap \{$weak facets$\}$, it suffices to show that $\{$weak facets$\} \subseteq \{$extreme functions$\}$.

Assume that $\pi$ is a weak facet, thus $\pi$ is minimal valid by \autoref{lemma:weak_facet_minimal}. We show that $\pi$ is extreme. For the sake of contradiction, suppose that $\pi$ is not extreme. By the assumption $\pi \in \mathcal{F}$, there exists a piecewise linear perturbation function $\bar\pi \not\equiv 0$ such that $\pi\pm\bar\pi$ are minimal valid functions. Furthermore, by \cite[Lemma 2.11]{igp_survey},  
we know that $\bar\pi$ is continuous, and $E(\pi)\subseteq E(\bar\pi)$. 
By taking the union of the breakpoints, 
we can define a common refinement, which will still be denoted by $\P$, of the complexes for $\pi$ and for $\bar\pi$. In other words, we may assume that $\pi$ and $\bar\pi$ are both continuous piecewise linear over $\P$.  
Since $\Delta\bar\pi\not\equiv0$, we may assume without loss of generality that 
$\Delta\bar\pi(x,y) > 0$ for some $(x,y)\in \verts(\Delta\P)$.
Define \[\epsilon = \min\left\lbrace \frac{\Delta\pi(x,y)}{\Delta\bar\pi(x,y)}\Bigst (x,y)\in \verts(\Delta\P),\;  \Delta\bar\pi(x,y) > 0\right\rbrace.\]
Notice that $\epsilon >0$, since $\Delta\pi \geq 0$ and $E(\pi)\subseteq E(\bar\pi)$.
Let $\pi'=\pi - \epsilon\bar\pi$. Then $\pi'$ is a bounded continuous function piecewise linear over $\P$, such that $\pi' \neq \pi$. 

The function $\pi'$ is subadditive, since $\Delta\pi'(x,y) \geq 0$ for each $(x,y)\in \verts(\Delta\P)$. As in the proof of \autoref{lemma:lipschitz-equiv-perturbation}, it can be shown that $\pi'$ is non-negative, $\pi'(0)=0$, $\pi'(f)=1$, and that $\pi'$ satisfies the symmetry condition. Therefore, $\pi'$ is a minimal valid function. Let $(u,v)$ be a vertex of $\Delta\P$ satisfying $\Delta\bar\pi(u,v) > 0$ and $\Delta\pi(u,v) = \epsilon\Delta\bar\pi(u,v)$. We know that $\Delta\pi'(u,v)=\Delta\pi(u,v)-\epsilon\Delta\bar\pi(u,v)=0$, hence $(u,v)\in E(\pi')$. However, $(u,v)\not\in E(\pi)$, since $\Delta\bar\pi(u,v)> 0$ implies that $\Delta\pi(u,v)\neq 0$.  Therefore, $E(\pi)\subsetneq E(\pi')$. By \autoref{lemma:weak_facet_minimal}(3), we have that $\pi$ is not a weak facet, a contradiction.
\end{proof}

\begin{remark}
  The theorem is stated for functions $\pi$ and $\bar\pi$ that are piecewise
  over the same complex~$\mathcal P$.  This is not a restriction because if we
  are given two complexes $\mathcal P$ and $\mathcal{\bar P}$, then we can
  define a new complex, the \emph{common refinement} of $\mathcal P$ and
  $\mathcal{\bar P}$, whose set of vertices is the union of those of
  $\mathcal P$ and $\mathcal{\bar P}$.
\end{remark}

\section{Separation of the notions in the discontinuous case}
\label{s:discontinuous_examples}

\subsection{Extreme, but not a weak facet}

The definitions of facets and weak facets fail to account for
additivities-in-the-limit, which are a crucial feature of the extremality test
for discontinuous functions.  
This allows us to separate the notion of extreme functions from the other two
notions.  Below we do this by observing that the discontinuous piecewise linear
extreme function $\psi=\sagefunc{hildebrand_discont_3_slope_1}()$, which appeared
above in \autoref{ex:hildebrand_discont_3_slope_1}, works as a separating
example.  

\begin{theorem}
\label{lemma:discontinuous_examples_1}
The function $\psi= \sage{hildebrand\_discont\_3\_slope\_1()}$ 
is a one-sided discontinuous piecewise linear function with rational
breakpoints that is extreme, but is neither a weak facet nor a facet.
\end{theorem}

\begin{figure}[tp]
\centering
\begin{minipage}{.49\textwidth}
\centering
\includegraphics[width=\linewidth]{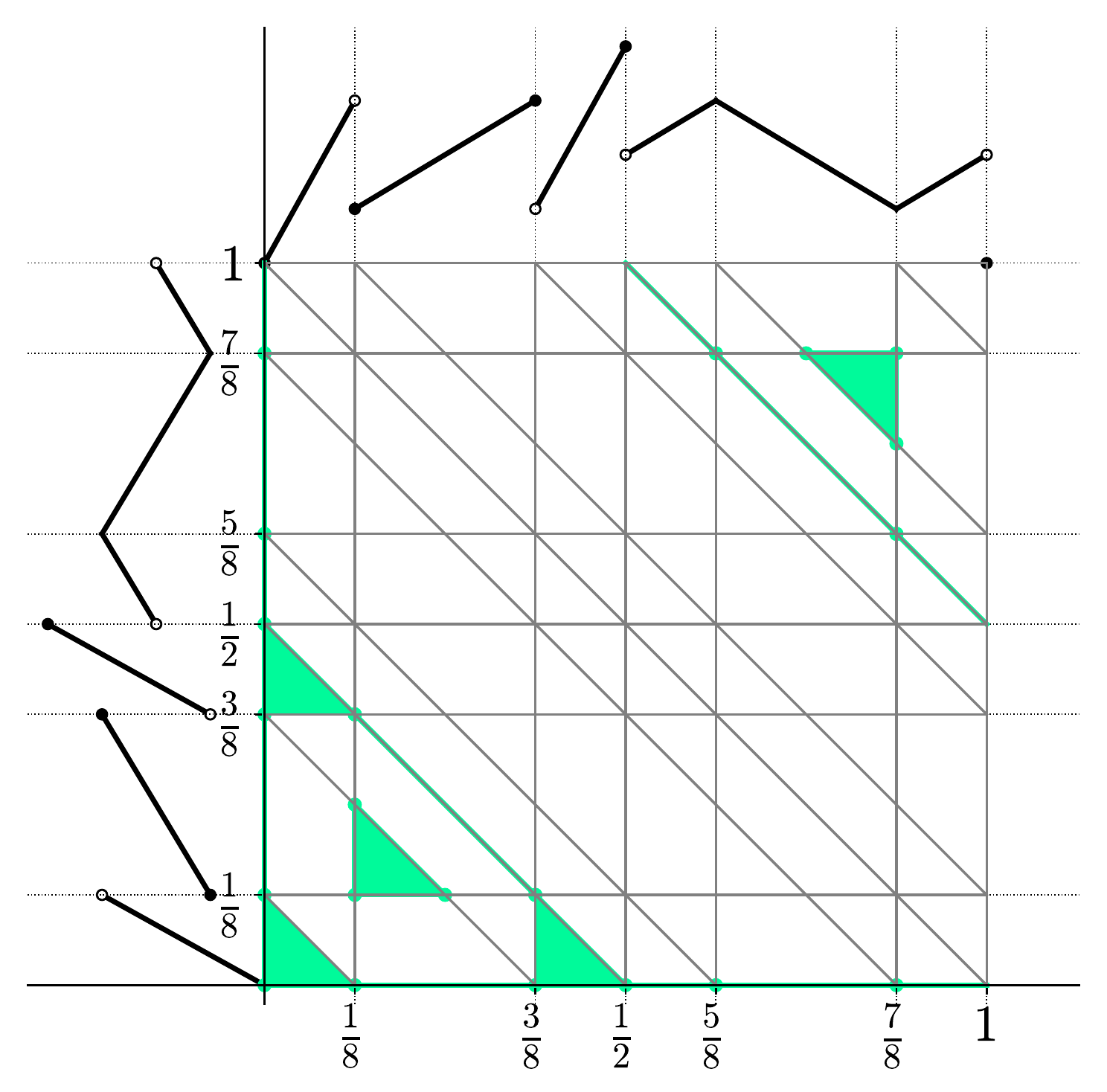}
\end{minipage}
\begin{minipage}{.49\textwidth}
\centering
\includegraphics[width=\linewidth]{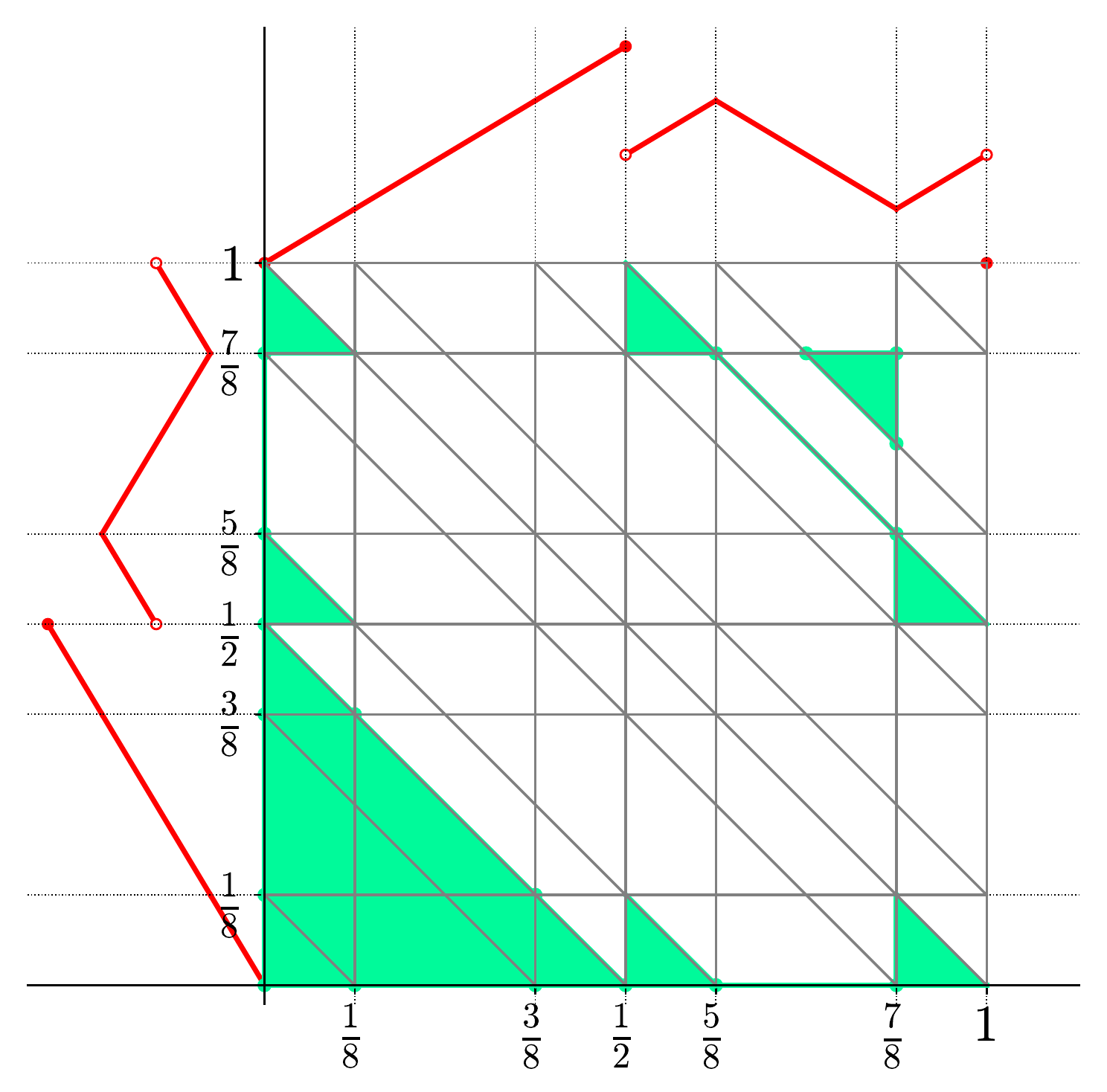}
\end{minipage}
\caption{Two diagrams of functions \sage{h} (\textit{graphs on the top and the left}) and polyhedral complexes $\Delta\P$ (\textit{gray solid lines}) with additive domains $E(\sage{h})$ (\textit{shaded in green}).
(\textit{Left,
black graph})
$\sage{h = hildebrand\_discont\_3\_slope\_1()}=\psi$. (\textit{Right, red graph}) \sage{h}
$=\psi'$ from the proof of \autoref{lemma:discontinuous_examples_1}.}
\label{fig:simple_E_pi_extreme_not_facet}
\end{figure}

\begin{figure}[tp]
\centering
\begin{minipage}{.49\textwidth}
\centering
\includegraphics[width=\linewidth]{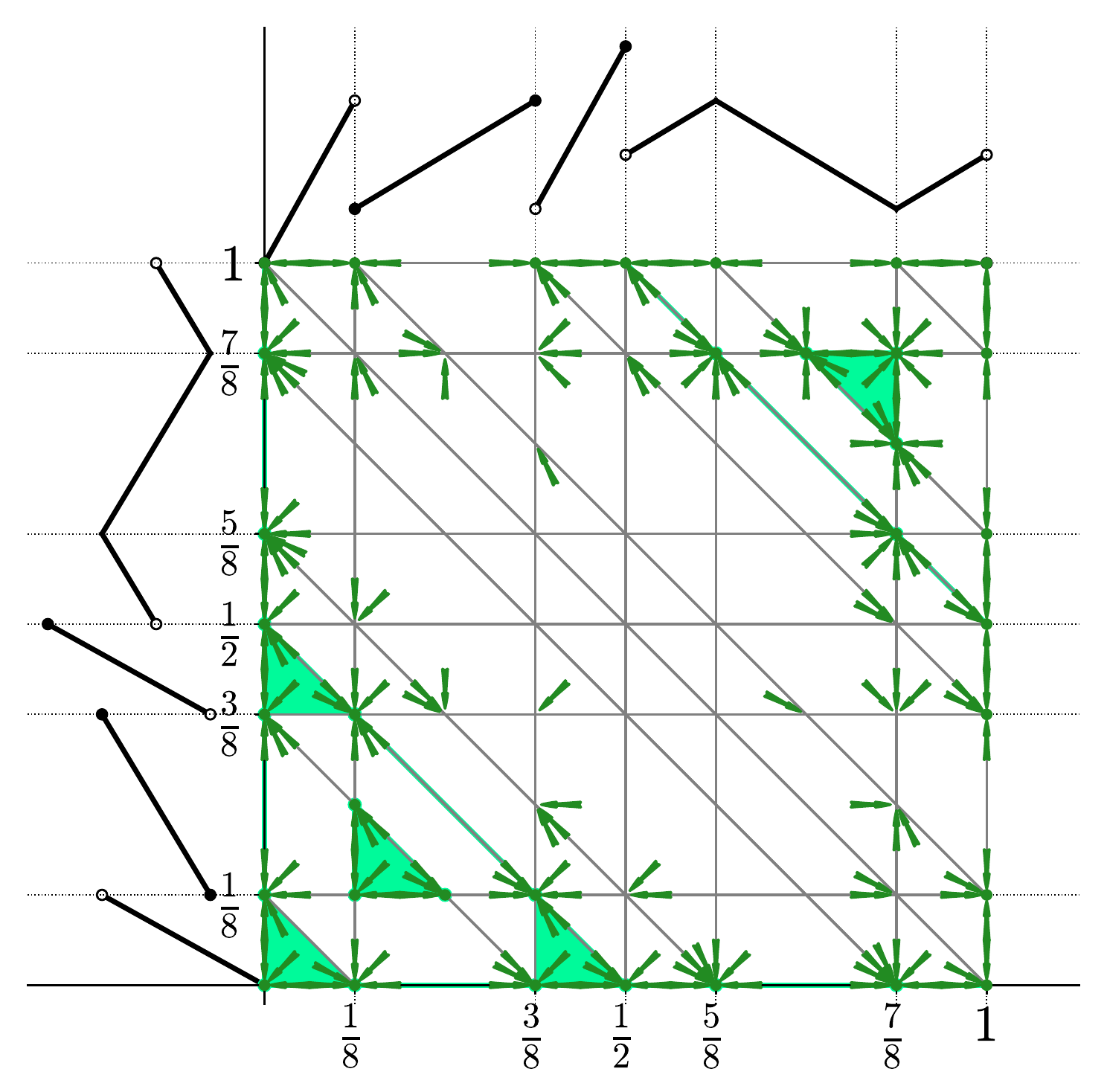}
\end{minipage}
\begin{minipage}{.49\textwidth}
\centering
\includegraphics[width=\linewidth]{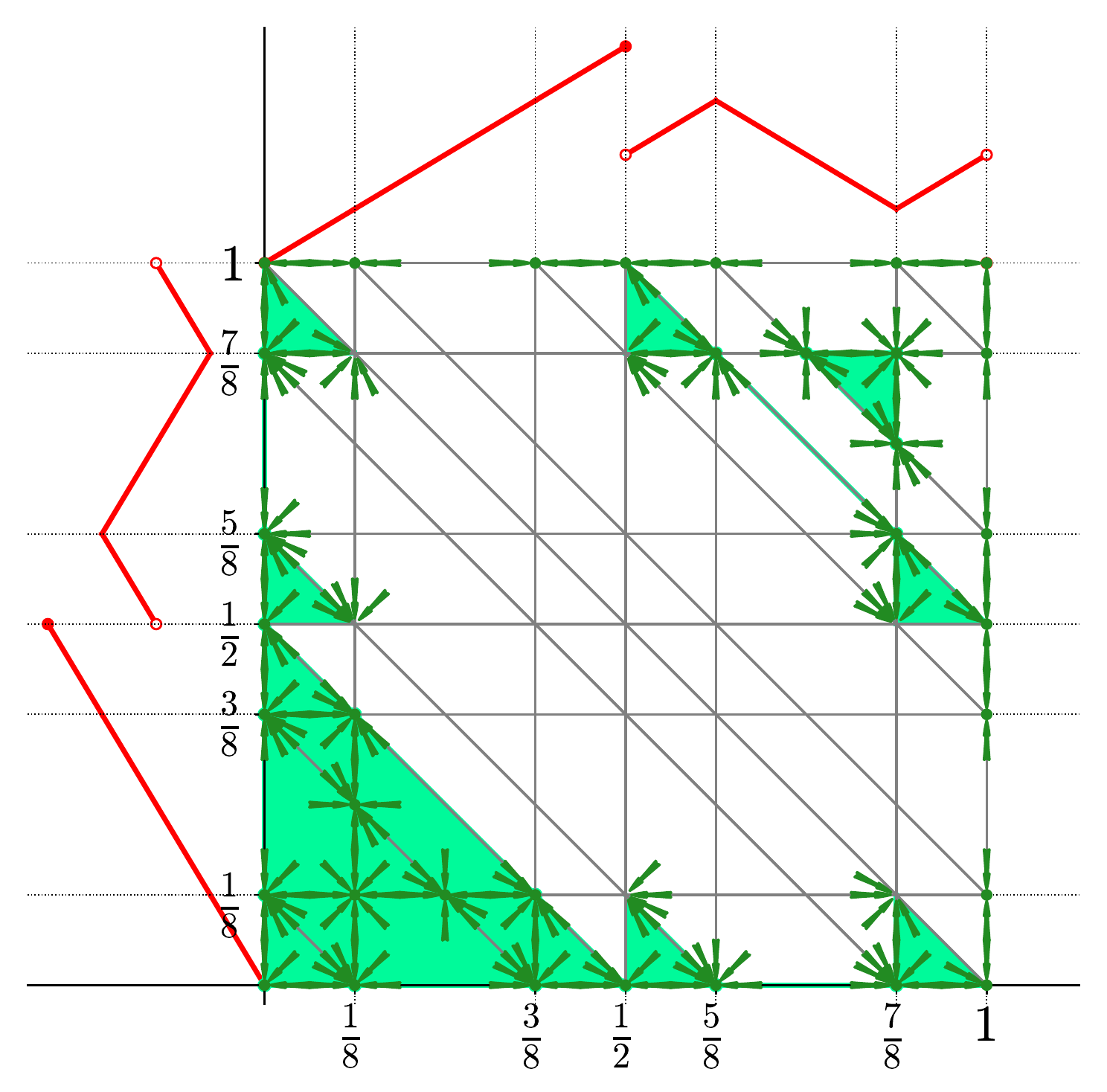}
\end{minipage}
\caption{Two diagrams of functions \sage{h} (\textit{graphs on the top
    and the left borders}) and polyhedral complexes $\Delta\P$ (\textit{gray solid
    lines}) with additive domains $E(\sage{h})$ (\textit{green triangles})
  and $E_\bullet(\sage{h}, \P)$ (\textit{green \iftrue arrows\else sectors\fi}).
  (\textit{Left})
  $\sage{h = hildebrand\_discont\_3\_slope\_1()}=\psi$. (\textit{Right})
  \sage{h = discontinuous\_facets\_paper\_example\_psi\_prime()}
  $=\psi'$ from the proof of \autoref{lemma:discontinuous_examples_1}.}
\label{fig:simple_E_pi_extreme_not_facet_with_cones}
\end{figure}

\begin{proof}
The function $\psi= \sage{hildebrand\_discont\_3\_slope\_1()}$ is extreme
(Hildebrand, 2013, unpublished, reported in \cite{igp_survey}).  
The extremality proof appears as \cite[Example 7.2]{hong-koeppe-zhou:software-paper};
it can also be verified using the software
\cite{cutgeneratingfunctionology:lastest-release}.%
\footnote{The command \sage{h =
    hildebrand\_discont\_3\_slope\_1(); extremality\_test(h)}
  carries out the verification.}
The function~$\psi$ is piecewise linear on a complex~$\P$, which is illustrated in
\autoref{fig:simple_E_pi_extreme_not_facet} (left). 
Consider the minimal valid function $\psi' =
\sage{discontinuous\_facets\_paper\_example\_psi\_prime}()$ defined by 
\[
\psi'(x)= 
      \begin{cases} 
        2x & \text{if } x \in [0, \frac12]; \\
        \psi(x) & \text{if } x \in (\frac12, 1).
      \end{cases}
\]
It can be considered as piecewise linear on the same complex~$\mathcal P$. 
Observe that $E(\psi)$ is a strict subset of $E(\psi')$. 
See \autoref{fig:simple_E_pi_extreme_not_facet} for an illustration of this
inclusion.
Thus, by \autoref{lemma:weak_facet_minimal}(3), the function $\psi$
is not a weak facet (nor a facet).
\end{proof}

We remark that there is no inclusion relation between the limit-additivities captured in the
set families $E_\bullet(\psi, \mathcal P)$ and $E_\bullet(\psi', \mathcal P)$, as
illustrated in the diagrams in \autoref{fig:simple_E_pi_extreme_not_facet_with_cones}.
We will explain these diagrams on an example only; see 
\cite{hong-koeppe-zhou:software-paper}, where these types of diagrams for
discontinuous piecewise linear functions were introduced,
for a full discussion.\footnote{The graphs in
  \autoref{fig:simple_E_pi_extreme_not_facet} can be reproduced with the
  command
  \sage{plot\underscore{}2d\underscore{}diagram\underscore{}additive\underscore{}domain\underscore{}sans\underscore{}limits(h)},
  those in \autoref{fig:simple_E_pi_extreme_not_facet_with_cones} with the
  command \sage{plot\underscore{}2d\underscore{}diagram\underscore{}additive\underscore{}domain\underscore{}sans\underscore{}limits(h)
    + plot\underscore{}2d\underscore{}diagram\underscore{}with\underscore{}cones(h)}.}
Consider $(\frac38,\frac38)$ as a vertex of the face~$F\in\Delta\P$ that is
the triangle to the northeast of it. The limit of $\Delta\psi$ within~$\relint(F)$ to
$(\frac38,\frac38)$ is $\lim_{(x,y)\to(\frac38^+,\frac38^+)} \Delta\psi(x,y) = 0$,
thus we have additivity in the limit. 
This is indicated by the 
\iftrue
green arrow from the
\else
green sector 
\fi
northeast of $(\frac38,\frac38)$.
But the corresponding limit of $\Delta\psi'$ is $
\lim_{(x,y)\to(\frac38^+,\frac38^+)} \Delta\psi'(x,y) > 0$.
As a result, the perturbation $\bar\psi = \psi' - \psi$ is not an effective
perturbation for~$\psi$:  For any $\epsilon>0$, the function $\psi - \epsilon \bar\psi$
violates subadditivity near $(\frac38,\frac38)$.  In fact, $\bar\psi$ does not
belong to the space $\bar{\Pi}^{E_{\bullet}}(\R,\Z)$ with $E_{\bullet} =
E_{\bullet}(\psi, \P)$.

\subsection{Weak facet, but not extreme}

The other separations appear to require more complicated
constructions.  Recently in \cite{koeppe-zhou:crazy-perturbation}, the authors
constructed a two-sided discontinuous 
piecewise linear minimal valid function,
$\pi = \sage{kzh\_minimal\_has\_only\_crazy\_perturbation\_1()}$, which is not extreme,
but which is not a convex combination of other piecewise linear minimal valid
functions; see \autoref{tab:kzh_minimal_has_only_crazy_perturbation_1} in
Appendix~\ref{s:kzh_minimal_has_only_crazy_perturbation_1} for the definition 
and \autoref{fig:complex_nf} for a graph.

This function has 40 breakpoints $0=x_0<x_1<\dots<x_{39}<x_{40}=1$ within $[0,1]$.
It has two \emph{special intervals} $(l, u) = (x_{17}, x_{18})$
and $(f-u, f-l) = (x_{19}, x_{20})$, where $f = x_{37} = \frac{4}{5}$, $l=\frac{219}{800}$,
$u=\frac{269}{800}$, on which every nonzero perturbation is
\emph{microperiodic}, 
namely invariant under the action of the dense additive group $T = \langle t_1,
t_2 \rangle_{\Z}$, where $t_1 = a_1 - a_0 = x_{10} - x_6 = \frac{77}{7752}\sqrt{2}$ and $t_2 = a_2 - a_0 = x_{13} - x_6 = \frac{77}{2584}$. 
Below we prove that it furnishes another separation.  

\begin{theorem}
\label{lemma:discontinuous_examples_2}
The function $\pi=
  \sage{kzh\_minimal\_has\_only\_crazy\_perturbation\_1()}$ 
  is a two-sided discontinuous piecewise linear function (with some irrational
  breakpoints) 
  that is not extreme (nor a facet), but is a weak facet.
\end{theorem}

\begin{proof}
  By
  \cite[{Theorem~\ref{crazy:th:kzh_minimal_has_only_crazy_perturbation_1}}]{koeppe-zhou:crazy-perturbation},
  we know that the function $\pi$ is minimal valid, but is not extreme. Let $\pi'$ be a minimal valid function such that $E(\pi) \subseteq E(\pi')$. We want to show that $E(\pi)= E(\pi')$. Consider $\bar\pi = \pi' - \pi$, which is a bounded $\Z$-periodic function satisfying that $E(\pi) \subseteq E(\bar\pi)$. As a difference of minimal valid functions, it satisfies the symmetry condition \begin{equation}
    \label{eq:bar-pi-symmetry-on-special}
    \bar\pi(x)+\bar\pi(y)=0 \quad\text{for $x, y\in\R$ such that $x+y = f$}
\end{equation}
as well as the conditions 
\begin{equation}
  \bar\pi(0)=\bar\pi(\tfrac{f}2)=\bar\pi(f)=\bar\pi(\tfrac{1+f}2) =\bar\pi(1)=0.
  \label{eq:fixed-to-zero}
\end{equation}
  We reuse parts of the proof of
  \cite[{Theorem~\ref{crazy:th:kzh_minimal_has_only_crazy_perturbation_1}}, Part
  (ii)]{koeppe-zhou:crazy-perturbation}, applying it to the
  perturbation~$\bar\pi$.


\smallbreak

  \DIFFPROTECT{First, as in the proof of \cite[{Theorem~\ref{crazy:th:kzh_minimal_has_only_crazy_perturbation_1}}, Part
  (ii)]{koeppe-zhou:crazy-perturbation}, we prove the following claim:
  \begin{enumerate}[\quad(i)]\setcounter{enumi}{-1}
  \item[\quad(o)]\label{bar-pi-pwl-outside-special}
    The function $\bar\pi$ is piecewise linear on~$\P$ outside of the
    special intervals, with unknown slopes $\bar c_1, \bar c_3\in\R$ on all
    intervals where $\pi$ has slopes $c_1$ and~$c_3$, respectively.
  \end{enumerate}
  See \autoref{tab:kzh_minimal_has_only_crazy_perturbation_1} for a list of the
  intervals. We reuse the computer-assisted proof in
  \cite[Appendix~C]{koeppe-zhou:crazy-perturbation} to prove Claim~(o).
  Because $E(\pi) \subseteq E(\bar\pi)$, if a two-dimensional face $F\in\Delta\P$
  satisfies 
  \begin{equation}
    \Delta\pi(x, y) = 0 \quad\text{for $(x,y)\in\relint(F)$,}
    \label{eq:additivity-sans-limits}
  \end{equation}
  then we also have
  \begin{equation}
    \Delta\bar\pi(x, y) = 0 \quad\text{for $(x,y)\in\relint(F).$}
    \label{eq:additivity-sans-limits-bar-pi}
  \end{equation}
  \autoref{tab:kzh_minimal_has_only_crazy_perturbation_1_faces_used_dim_2} shows a list of faces $F \in \Delta\P$ with this propety.
  Our proof repeatedly applies the Gomory--Johnson Interval Lemma
  in the form of \cite[Theorem 4.3]{igp_survey}
  to these faces.  (This version of the theorem only requires
  the boundedness of the function~$\bar\pi$; this is contrast to the proof in
  \cite{koeppe-zhou:crazy-perturbation}. The latter uses a version that is stated for 
  effective perturbations only.) 
  By the theorem, $\bar\pi$ is affine linear with the same slope on the open
  intervals $\intr(p_i(F))$ for $i=1,2,3$.

  Then the proof considers the edges $F \in \Delta\P$ that satisfy 
  \eqref{eq:additivity-sans-limits}
  shown in \autoref{tab:kzh_minimal_has_only_crazy_perturbation_1_faces_used_dim_1}.
  Let $\{i,j\} \subset \{1,2,3\}$ such that $p_i(F)$ and $p_j(F)$ are proper intervals. 
  Let $L \subseteq F$ be a line segment such that $\bar\pi$ is affine linear on $p_i(L)$.
  Then by \eqref{eq:additivity-sans-limits-bar-pi}, $\bar\pi$ is also affine linear with the same slope on $p_j(L)$. 
  (There is another difference to the proof in \cite{koeppe-zhou:crazy-perturbation}: 
  property \eqref{eq:additivity-sans-limits} is more
  specific than the hypothesis of \cite[Theorem
  3.3]{koeppe-zhou:crazy-perturbation}.  The latter
  only requires limit-additivities
  $ \Delta\pi_{F'}(x, y) = 0$ for $(x,y)\in\relint(F)$
  where $F' \supseteq F$ is an enclosing face.
  This distinction is crucial because we
  have no control over the limit-additivities of $\bar\pi$.)}
\smallbreak

\DIFFPROTECT{Next we establish the following stronger claim:
  \begin{enumerate}[\quad(i)]\setcounter{enumi}{0}
  \item\label{li:bar-pi-zero-outside-special}
    We have $\bar\pi(x)=0$ for $x \not\in (l, u) \cup (f-u, f-l)$. 
  \end{enumerate}
  Our proof is again similar to the one of
  \cite[{Theorem~\ref{crazy:th:kzh_minimal_has_only_crazy_perturbation_1}},
  Part (ii)]{koeppe-zhou:crazy-perturbation},
  but in contrast to that, 
  we consider only the restriction of~$\bar\pi$ to 
  $$[0,l] \cup [u, f-u] \cup [f-l, 1],$$ 
  where the function is piecewise linear by
  (o)
  .  The restricted function is determined by a finite system of
  parameters as follows: two slope parameters $\bar c_1$ and $\bar c_3$,
  19~parameters that determine the function value $\bar\pi(x_i)$ at each
  breakpoint, and 18~parameters that determine the midpoint function value
  $\bar\pi(\frac{x_i+x_{i+1}}2)$ on each interval of~$\P$ except for the special intervals.  (Here we used the
  symmetry condition~\eqref{eq:bar-pi-symmetry-on-special}, as well as the conditions
  \eqref{eq:fixed-to-zero} to reduce the number of parameters.)}
\DIFFPROTECT{We set up a finite linear system of equations that expresses the
  additivity relations \eqref{eq:additivity-sans-limits-bar-pi} for faces $F$
  that satisfy \eqref{eq:additivity-sans-limits}.  
  We do this by writing equations
  $\Delta\bar\pi_F(x, y) = 0$ for $(x, y) \in \verts(F)$ for these faces~$F$.
  The system has full rank; a regular $39\times 39$ subsystem 
  is shown in
  Tables~\ref{tab:kzh_minimal_has_only_crazy_perturbation_1_faces_of_vertices_used}
  and~\ref{tab:kzh_minimal_has_only_crazy_perturbation_finite_system}.
  Therefore the unique solution of the system is~$0$, and
  Claim~(\ref{li:bar-pi-zero-outside-special}) is proved.}
\smallbreak

\DIFFPROTECT{Next, we show that
\begin{enumerate}[\quad(i)]\setcounter{enumi}{1}
  \item\label{li:bar-pi-constant-on-cosets-on-special}
    $\bar\pi$ is constant on each coset $\bar x + T \in \R/T$ on the special interval $(l, u)$, and likewise on 
    the special interval $(f-u, f-l)$. 
  \end{enumerate}
The
function $\bar\pi$ satisfies the additivity relations
\eqref{eq:additivity-sans-limits-bar-pi} from the faces
$
F(\{a_i\}, [l, u],\allowbreak [f-u, f-l])$ for $i=0,1,2$, where $a_0 = x_6$,
$a_1 = x_{10} = a_0 + t_1 = \textstyle\frac{77}{7752}\sqrt{2} +
\tfrac{19}{100}$, and $a_2 = x_{13} = a_0 + t_2$. These faces appear in
\autoref{tab:kzh_minimal_has_only_crazy_perturbation_1_delta_pi_dim_1}; see
also \autoref{fig:complex_nf}. 
Let $\bar x$ be an arbitrary real number. 
Then there exists a point $\hat x\in(l, u)$ such that $\bar x - \hat x \in T$
and $\hat x\pm t_i \in (l, u)$.  
Then $\bar\pi$ and $\hat x$ satisfy the hypothesis of
\cite[Lemma~B.1]{koeppe-zhou:crazy-perturbation}.  Writing $\bar x - \hat x = \lambda_1 t_1 + \lambda_2 t_2$ 
for some $\lambda_1, \lambda_2\in\Z$, the lemma gives $\bar\pi(\bar x) - \bar\pi(\hat x) = \sum_{i=1}^2\lambda_i(\bar\pi(a_i)-\bar\pi(a_0)$.  Using $\bar\pi(a_i) = 0$, as the points $a_i$ lie
outside of the special intervals, we obtain $\bar\pi(\bar x) = \bar\pi(\hat x)$.  Using the
symmetry relation given by \eqref{eq:additivity-sans-limits-bar-pi} for the face $
F([l, u], [f-u, f-l], \{f\})$, we obtain that $\bar\pi$ is constant on the set~$(f-u, f-l) \cap (f-\bar x + T)$ as well.
}
\smallbreak

Next, using the face $F([l, u], [l, u], \{l+u\})$, which satisfies~\eqref{eq:additivity-sans-limits},
and $\bar\pi(l+u) = 0$ from (\ref{li:bar-pi-zero-outside-special}) because $l+u$ lies outside the special intervals,
we obtain that
\begin{enumerate}[\quad(i)]\setcounter{enumi}{2}
\item\label{li:bar-pi-additivities-on-special-at-z=l+u--pre}
  $\bar\pi(x)+\bar\pi(y)=0$ for $x, y \in (l,u)$ such that $x+y = l+u$.
\end{enumerate}
Together with (\ref{li:bar-pi-constant-on-cosets-on-special}), 
we obtain that 
\begin{enumerate}[\quad(i)]\setcounter{enumi}{3}
\item\label{li:bar-pi-additivities-on-special-at-z=l+u}
  $\bar\pi(x)+\bar\pi(y)=0$ for $x, y \in (l,u)$ such that $x+y \in (l+u) + T$.
\end{enumerate}
\smallbreak

We now show that $\bar\pi$ also satisfies the following condition:
\begin{enumerate}[\quad(i)]\setcounter{enumi}{4}
\item\label{li:bar-pi-abs-leq-s-on-special}
  $\left|\bar\pi(x)\right| \leq s$ for all $x \in (l, u) \cup (f-u, f-l)$, 
\end{enumerate}
where
\begin{equation}\label{eq:magic-s}
  s = \pi(x_{39}^-)+\pi(1+l-x_{39})-\pi(l) = \tfrac{19}{23998}.
\end{equation}
Indeed, by (\ref{li:bar-pi-additivities-on-special-at-z=l+u--pre}) and \eqref{eq:bar-pi-symmetry-on-special}, it suffices to show that for any $x \in (l, u)$, we
have $\bar\pi(x) \geq -s$. Suppose, for the sake of contradiction, that there
is $\bar x \in (l, u)$ such that $\bar\pi(\bar x) < -s$. Since the group $T$
is dense in $\R$, we can find $x \in (l, u)$ such that $x \in \bar x + T$ and
$x$ is arbitrarily close to $1+l - x_{39}$. We choose $x$ so that $\delta = x
- (1+l - x_{39}) \in (0, \frac{-s-\bar\pi(\bar x)}{c_2 - c_3})$, where $c_2$
and $c_3$ denote the slope of $\pi$ on the pieces $(l, u)$ and $(0, x_1)$,
respectively.  
See 
\DIFFPROTECT{\autoref{tab:kzh_minimal_has_only_crazy_perturbation_1}}
for the concrete values of the parameters. 
Let $y = 1+l -x$. Then $y = x_{39} -\delta$. It follows from (i) that $\bar\pi(y)=0$ and $\bar\pi (x+y) = \bar\pi(l)=0$. Now consider $\Delta\pi'(x,y) = \pi'(x)+\pi'(y)-\pi'(x+y)$, where
\[
\begin{array}{r@{\;}l@{\;}l}
\pi'(x)&= \bar\pi(x) + \pi(x)&=\bar\pi(x) +\pi(1+l-x_{39})+\delta c_2;\\
\pi'(y)&=\pi(y)&=\pi(x_{39}^-) - \delta c_3;\\
\pi'(x+y)&=\pi(x+y)&=\pi(l).  
\end{array}
\]
Since $x - \bar x \in T$, the condition (\ref{li:bar-pi-constant-on-cosets-on-special}) implies that $\bar\pi(x) = \bar\pi(\bar x)$.
We have
\begin{align*}
\Delta\pi'(x,y) &= \bar\pi(\bar x)+[\pi(1+l-x_{39})+ \pi(x_{39}^-)-\pi(l)]+\delta(c_2 - c_3) \\
& =\bar\pi(\bar x) + s + \delta(c_2 - c_3)
  < 0,
\end{align*}
a contradiction to the subadditivity of $\pi'$. Therefore, $\bar\pi$ satisfies
condition~(\ref{li:bar-pi-abs-leq-s-on-special}).\smallbreak

\begin{figure}[tp]
\centering
\includegraphics[width=\linewidth]{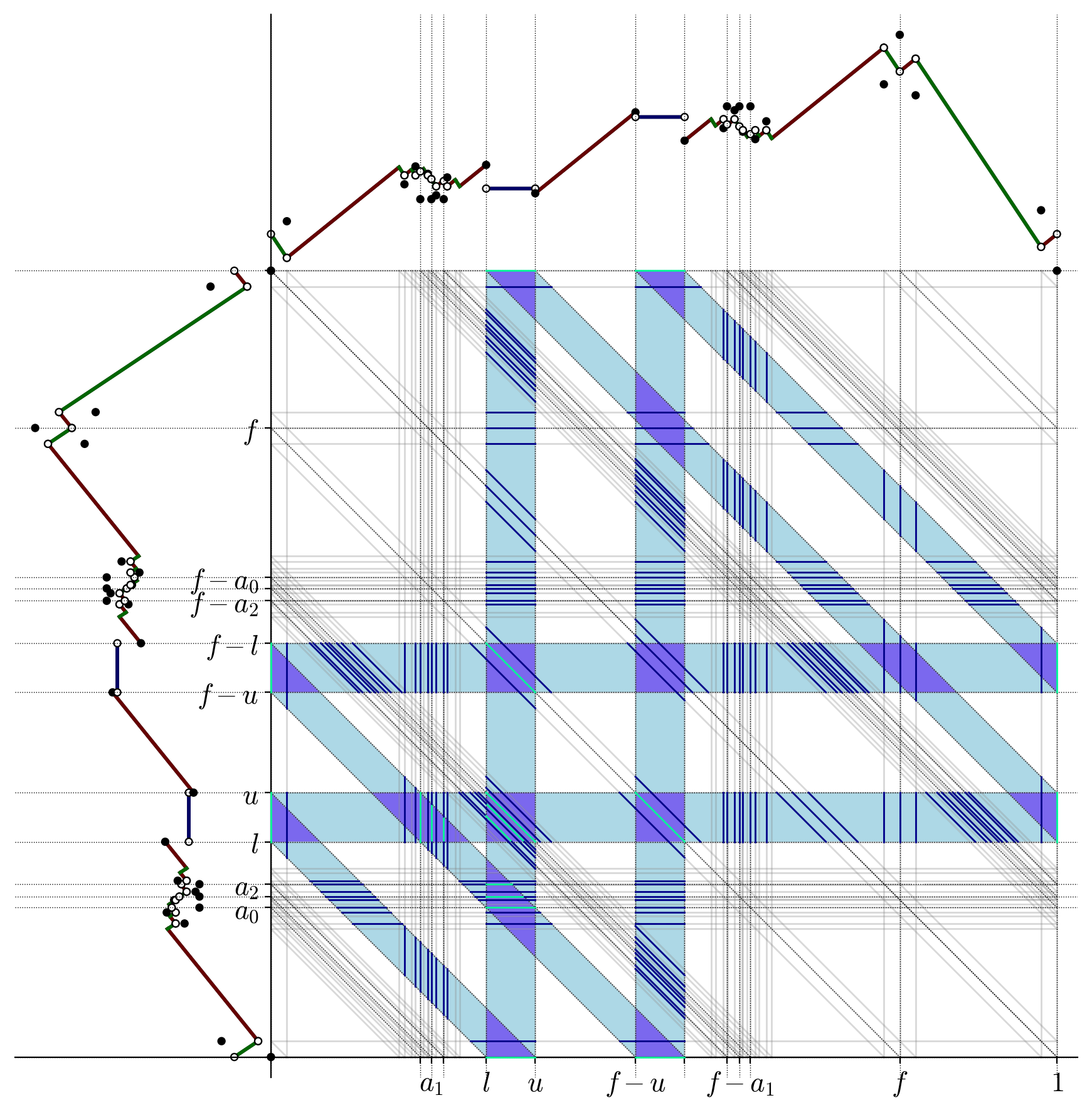}
\caption{Diagram of the polyhedral complex $\Delta\P$ of the function $\pi=
  \sage{kzh\_minimal\_has\_only\_crazy\_perturbation\_1()}$ (shown at the top
  and left borders), where two-dimensional faces $F$
  are color-coded according to the values $n_F$: $n_F=0$ (\emph{white}),
  $n_F=1$ (\emph{light blue}), $n_F=2$ (\emph{medium lavender blue}).
  One-dimensional faces $F$ with $n_F > 0$ are shown in (a) \emph{light green} if $\Delta\pi_F(u,v) = 0$,
  (b) \emph{dark blue} if $\Delta\pi_F(u,v) \geq n_F \cdot s$ for $(u,v)\in\verts(F)$.
}
\label{fig:complex_nf}
\end{figure}
Let $F$ be a face of $\Delta\P$. Denote by $n_F \in \{0, 1, 2\}$ the number
of projections $p_i(\relint(F))$ for $i=1, 2, 3$ that intersect with  $(l, u)
\cup (f-u, f-l)$.  See \autoref{fig:complex_nf}; note that there is no face $F$ with $n_F=3$.
It follows from the conditions (\ref{li:bar-pi-zero-outside-special}) and~(\ref{li:bar-pi-abs-leq-s-on-special}) that
\[\left|\Delta\bar\pi(x,y)\right| \leq n_F \cdot s \quad\text{ for any } (x,y) \in \relint(F).\]
Moreover, we make the following claim.
\begin{enumerate}[\quad(i)]\setcounter{enumi}{5}
\item\label{claim:n_F}
  If $F\in \Delta\P$ has $n_F \neq 0$, then either
  \begin{enumerate}[(a)]
  \item[(a)] $\Delta\pi_F(u,v)=0$ for all $(u,v) \in \verts(F)$, or
  \item[(b)] $\Delta\pi_F(u,v)\geq n_F \cdot s$ for all $(u,v) \in \verts(F)$, and the inequality is strict for at least one vertex.
  \end{enumerate}
\end{enumerate}
We have verified this claim computationally (using exact computations) in our
software, by enumerating all faces~$F$ of $\Delta\P$ with
$n_F>0$.\footnote{The enumeration is done by the function
  \sage{generate\_faces\_with\_projections\_intersecting}. 
  A fully automatic verification is carried out by the
  command 
  \sage{kzh\_minimal\_has\_only\_crazy\_perturbation\_1\_check\_subadditivity\_slacks()}.}
We provide the relevant data of the function in
Appendix~\ref{s:kzh_minimal_has_only_crazy_perturbation_1} (Tables~\ref{tab:kzh_minimal_has_only_crazy_perturbation_1_delta_pi_dim_1} and
\ref{tab:kzh_minimal_has_only_crazy_perturbation_1_delta_pi_dim_2}) for
the reader's reference and for archival purposes.\smallbreak

Next, we show the following simple corollary of Claim~(\ref{claim:n_F}):
\begin{enumerate}[\quad(i)]\setcounter{enumi}{6}
\item\label{claim:n_F-consequence}
  If $F\in \Delta\P$ has $n_F \neq 0$, then either
  \begin{enumerate}[(a)]
  \item[(a)] $\Delta\pi(x,y)=0$ for all $(x,y) \in \relint(F)$, or
  \item[(b)] $\Delta\pi(x,y) > n_F \cdot s$ for all $(x,y) \in \relint(F)$.
  \end{enumerate}
\end{enumerate}
To prove this, assume that $n_F \neq 0$.  Since $\Delta\pi_F$ is affine linear
on $F$, $\Delta\pi(x,y)$ for $(x, y)\in\relint(F)$ is a strict convex
combination of $\{\,\Delta\pi_F(u, v) \st 
(u, v) \in \verts(F)\,\}$.  As at least one of the inequalities
$\Delta\pi_F(u,v)\geq n_F \cdot s$ is strict, (b) follows.\smallbreak

Finally, we prove the following claim:
\begin{enumerate}[\quad(i)]\setcounter{enumi}{7}
\item\label{li:delta-pi-positive-implies-delta-pi'-positive}
  For $(x, y) \in \R^2$ such that $\Delta\pi(x,y)>0$, we have $\Delta\pi'(x,y) > 0$.
\end{enumerate}
To prove this, consider the (unique) face $F \in \Delta\P$ such that $(x,y)
\in \relint(F)$.  If $n_F = 0$, then $\Delta\bar\pi(x,y)=0$, and hence
$\Delta\pi'(x,y)=\Delta\pi(x,y)>0$. Otherwise, because 
we have $\Delta\pi(x,y) > 0$ by assumption, the above case (b) applies, and
hence 
$\Delta\pi'(x,y)=\Delta\pi(x,y)+\Delta\bar\pi(x,y) > 0$ holds when $n_F \neq 0$ as well. 
\smallbreak

We obtain that $E(\pi')\subseteq E(\pi)$. This, together with the assumption $E(\pi)\subseteq E(\pi')$, implies that $E(\pi)=  E(\pi')$. 
We conclude, by \autoref{lemma:weak_facet_minimal}(3), that $\pi$ is a weak
facet.
\end{proof}

\subsection{Extreme and weak facet, but not a facet}

For the remaining separation, we construct an extreme function $\pilifted$ as
follows.  In
\cite[{Theorem~\ref{crazy:th:kzh_minimal_has_only_crazy_perturbation_1}}]{koeppe-zhou:crazy-perturbation},
the authors showed that
$\pi =
\sage{kzh\_minimal\_has\_only\_}$\allowbreak$\sage{crazy\_}$\allowbreak$\sage{perturbation\_1()}$
admits an effective locally microperiodic perturbation that is supported on
the cosets $l + T$, $u + T$ of the group~$T$ on the special interval $(l,u)$
and, equivariantly, on the cosets $f - l + T$, $f - u + T$ on the special
interval $(f-u, l-u)$.

We perturb 
the function $\pi$ instead on infinitely
many (almost all) cosets of the group~$T$ on the two
special intervals as follows.  
\DIFFPROTECT{Consider the involution (point reflection) $\rho_{l+u}\colon x \mapsto l+u - x = f-a_0 - x$, which has the unique fixed point $\frac{l+u}2$.
Because $\rho_{l+u}(x + t) = \rho_{l+u}(x) - t$ for $x\in\R$ and $t \in T$, 
the involution can be considered as a map from the quotient $\R/T$ (whose elements are the cosets of~$T$) 
to itself.
The set of fixed points of $\R/T$ under this map is
\begin{align*}
  C &= 
      \bigl\lbrace \tfrac{l+u}{2}+T, \;  \tfrac{l+u-t_1}{2}+T, \; \tfrac{l+u-t_2}{2}+T, \; \tfrac{l+u-(t_1+t_2)}{2}+T\bigr\rbrace.
  \\
  \intertext{The remaining elements of $\R/T$ are paired by the involution into two-element orbits $\{ x, \rho_{l+u}(x) \}$.  
  Fix a choice function $c^+$ that maps each of the two-element sets $\{ x, \rho_{l+u}(x) \} \subset \R/T$ to one of its two elements. 
  (We remark that the existence of such a choice function does \emph{not} depend on the axiom of choice because 
  the sets in question are finite.)
  Then define}
  C^+ &=\bigl\lbrace\, c^+(\{x, \rho_{l+u}(x) \}) \in \R/T \bigst x \in \R/T,\, x \not\in C \,\bigr\rbrace.
\end{align*}
Using these sets, we define for every $x \in [0,1]$}
\begin{equation}\label{eq:lifted_crazy_example}
      \pilifted(x)= 
      \begin{cases} 
        \pi(x) & \text{if } x \not\in (l, u) \cup (f-u, f-l) \text{, or} \\
        & \text{if } x \in (l, u) \text{ such that } x + T \in C \text{, or} \\
        & \text{if } x \in (f-u, f-l) \text{ such that } f - x + T \in  C; \\
        \pi(x) + s & \text{if } x \in (l, u) \text{ such that } x + T \in C^+  \text{, or} \\
        & \text{if } x \in (f-u, f-l) \text{ such that } f - x + T \in C^+; \\
        \pi(x) - s  & \text{otherwise},
      \end{cases}
    \end{equation}
    where $s$ is the constant defined in \eqref{eq:magic-s} in the proof of
    \autoref{lemma:discontinuous_examples_2}.
    We extend this function to~$\R$ by setting $\pilifted(x+z) = \pilifted(x)$ for $x\in\R$ and $z\in\Z$.

\begin{theorem}
  \label{lemma:discontinuous_examples_3}
  The function $\pilifted =
  \sage{kzh\_extreme\_and\_weak\_facet\_but\_not\_facet}()$\footnote{The
    authors thank Jiawei Wang for his help with implementing this function in
    the software.}
  defined in
  \eqref {eq:lifted_crazy_example} is a two-sided discontinuous,
  non--piecewise linear function that is extreme and a
  weak facet, but is not a facet.
\end{theorem} 

\begin{proof}
  Let $\bar\pi = \pilifted - \pi$
  . By definition of $\pilifted$, the function $\bar\pi$ is periodic modulo
  1. Moreover, it satisfies the symmetry condition
  \eqref{eq:bar-pi-symmetry-on-special}, the
  conditions~\eqref{eq:fixed-to-zero}, as well as the conditions
  (\ref{li:bar-pi-zero-outside-special}) to
  (\ref{li:bar-pi-abs-leq-s-on-special}) in the proof of
  \autoref{lemma:discontinuous_examples_2}. 
  We claim that $\pilifted$ is subadditive. 
  To this end, recall the notation $n_F$ from the proof of \autoref{lemma:discontinuous_examples_2}.

  For all faces~$F\in\Delta\P$ with $n_F=0$, 
  because $\pilifted$ equals $\pi$ outside of the special intervals, 
  we have $\Delta\bar\pi_F(x,y) = 0$
  for $(x, y) \in \relint(F)$, and thus $\Delta\hat\pi_F(x,y) = \Delta\pi_F(x,
  y)$ for $(x, y)\in\relint(F)$.

  Next, consider the faces $F$ with $n_F>0$.
  By Claim~(\ref{claim:n_F-consequence}) from
  \autoref{lemma:discontinuous_examples_2}, we either have (a) $\Delta\pi_F(x, y)
  = 0$ for all $(x, y)\in \relint(F)$, or (b) $\Delta\pi_F(x,y) > n_F\cdot s$ for all $(x,
  y)\in \relint(F)$.  

  From \autoref{tab:kzh_minimal_has_only_crazy_perturbation_1_delta_pi_dim_1} (see also
  \autoref{fig:complex_nf}), we see that these faces satisfying (a) are exactly the following
  (up to replacing $F(I, J, K)$ by $F(J, I, K)$ and up to $\Z$-periodicity): 
  \begin{enumerate}[(1)]
  \item $F = F([l, u], \{a_i\}, [f-u, f-l])$ for $i=1,2,3$.  Denoting $t_0 = 0$ for
    convenience, we have $a_i = a_0 + t_i$
    , where $t_i \in T$.
    Fix $i$ and let $(x, a_i) \in \relint(F)$, 
    so $x \in (l, u)$ and $x + a_i \in (f-u, f-l)$.  
    Because $\bar\pi(a_i) = 0$, as $a_i$ lies outside of the special
    intervals, and $(l, u) \ni f - (x + a_i) = f-a_0 - x - t_i = \rho_{l+u}(x) -
    t_i \in \rho_{l+u}(x) + T$,
    we have $\Delta\bar\pi(x, a_i) = \bar\pi(x) +
    \bar\pi(a_i) -\bar\pi(x + a_i) = \bar\pi(x) + \bar\pi(f - (x + a_i)) =
    \bar\pi(x) + \bar\pi(\rho_{l+u}(x)) = 0$. 

  \item $F = F([l, u], [l, u], \{f-a_i\})$ for $i=1,2,3$.  Again fix~$i$ and let
    $(x, y) \in \relint(F)$, so $x + y = f-a_i = f - a_0 - t_i$ and $x,
    y\in(l, u)$.  Then $\Delta\bar\pi(x, y) = 
    \bar\pi(x) + \bar\pi((f - a_0) - x - t_i) - \bar\pi(f-a_i) = \bar\pi(x) +
    \bar\pi(\rho_{l+u}(x)) = 0$.

  \item $F = F([l, u], [f-u, f-l], \{f\})$.  Then, by the symmetry
    condition~\eqref{eq:bar-pi-symmetry-on-special}, we have $\Delta\bar\pi(x,y) = 0$.
  \item $F(\{0\}, [l, u], [l, u])$ and $F(\{0\}, [f-u, f-l], [f-u,
    f-l])$. Here $\Delta\bar\pi(x,y) = 0$ trivially.
  \end{enumerate}
  Again, we conclude that $\Delta\hat\pi_F(x,y) =
  \Delta\pi_F(x, y)$ for $(x, y) \in \relint(F)$.

  Finally, consider the faces $F$ with $n_F>0$ that satisfy (b), i.e.,
  $\Delta\pi_F(x, y) > n_F\cdot s$ for $(x,
  y)\in \relint(F)$.   Because  $|\Delta\bar\pi_F(x,y)| \leq
  n_F\cdot s$, we have $\Delta\hat\pi_F(x, y) > 0$ for $(x, y)\in \relint(F)$.

  Hence, $\hat\pi$ is subadditive as claimed, and therefore a minimal valid
  function, and in fact $E(\pilifted)=E(\pi)$. 

  Let $\pi'$ be a minimal valid function such that $E(\pilifted)\subseteq E(\pi')$. Then, as shown in the proof of \autoref{lemma:discontinuous_examples_2}, we have $E(\pilifted) = E(\pi')$. It follows from \autoref{lemma:weak_facet_minimal}(3) that $\pilifted$ is a weak facet. However, the function $\pilifted$ is not a facet, since $E(\pilifted)=E(\pi)$ but $\pilifted \neq \pi$. Next, we show that $\pilifted$ is an extreme function. 

Suppose 
that $\pilifted$ can be written as $\pilifted = \frac12 (\pi^1+\pi^2)$, where $\pi^1, \pi^2$ are minimal valid functions. 
Then $E(\pilifted) \subseteq E(\pi^1)$ and $E(\pilifted) \subseteq E(\pi^2)$. 
Let $\bar\pi^1= \pi^1-\pi$ and $\bar\pi^2=\pi^2-\pi$. We have that  $E(\pi)
\subseteq E(\bar\pi^1)$ and $E(\pi) \subseteq E(\bar\pi^2)$. Hence, as shown
in the proof of \autoref{lemma:discontinuous_examples_2}, $\bar\pi^1$ and
$\bar\pi^2$ satisfy the symmetry
condition \eqref{eq:bar-pi-symmetry-on-special} and the conditions (\ref{li:bar-pi-zero-outside-special}) to (\ref{li:bar-pi-abs-leq-s-on-special}). We will show that $\bar\pi^1=\bar\pi^2$.

For $x\not\in (l,u)\cup (f-u, f-l)$, we have $\bar\pi^i(x)=0$ ($i=1, 2$) by
condition (\ref{li:bar-pi-zero-outside-special}). It remains to prove that $\bar\pi^1(x)=\bar\pi^2(x)$ for $x \in
(l,u)\cup (f-u, f-l)$. By the symmetry condition \eqref{eq:bar-pi-symmetry-on-special}, it suffices to consider
$x \in (l, u)$. We distinguish three cases. If $x + T\in C$, then condition
(\ref{li:bar-pi-additivities-on-special-at-z=l+u}) implies $\bar\pi^i(x)=0$ ($i=1, 2$). If $x + T\in C^+$, then
$\bar\pi(x)=s$ by definition. Notice that $\bar\pi^1 +\bar\pi^2 =
\pi^1+\pi^2-2\pi =2\pilifted-2\pi = 2\bar\pi$, and that
$\bar\pi^i(x) \leq s$ ($i=1, 2$) by condition (\ref{li:bar-pi-abs-leq-s-on-special}). We have $\bar\pi^i(x)=s$
($i=1, 2$) in this case. If $x + T\not\in C $ and $x + T \not\in C^+$, then $\bar\pi(x)=-s$, and hence $\bar\pi^i(x)=-s$ ($i=1, 2$). Therefore, $\bar\pi^1=\bar\pi^2$ and $\pi^1=\pi^2$, which proves that the function $\pilifted$ is extreme.
\end{proof}

\section{Conclusion}
\label{s:conclusion}

As a conclusion to our paper, we discuss the three notions relative to
subspaces of functions.  

Gomory and Johnson introduced the notion of facets in \cite{tspace} in a
setting in which valid functions, by definition, are continuous functions.
Following the discussion in \cite{bccz08222222}, a continuous valid function
$\pi$ is defined to be a \emph{facet in the sense of Gomory--Johnson} if
$P(\pi) \subseteq P(\pi')$ implies $\pi' = \pi$ for every \emph{continuous}
valid function $\pi'$.  As remarked in \cite{bccz08222222}, every continuous
facet is also a facet in the sense of Gomory--Johnson.  We have a partial
converse as follows.

\begin{corollary}
  \label{cor:facet-in-the-sense-of-GJ}
  Every continuous piecewise linear function (not necessarily with rational
  breakpoints) that is a facet in the sense of Gomory--Johnson is also a
  facet.
\end{corollary}
\begin{proof}
  Let $\pi$ be a continuous piecewise linear minimal valid function that is
  not a facet. Then $\pi$ is not an extreme function.  Thus $\pi = \frac12(\pi^1
  + \pi^2)$ with some minimal valid functions $\pi^1, \pi^2 \neq \pi$, which
  are Lipschitz continuous by \cite[Lemma 2.11\,(iv)]{igp_survey} and satisfy
  $E(\pi) \subseteq E(\pi^i)$ by \cite[Lemma 2.11\,(ii)]{igp_survey}.
  Setting $\pi' = \pi^1$, it follows from \autoref{lemma:P_pi_and_E_pi} that
  $P(\pi) \subseteq P(\pi')$. Therefore $\pi$ is not a facet in the sense of
  Gomory--Johnson.
\end{proof}

\begin{openquestion}
  Is every facet in the sense of Gomory--Johnson a facet?
\end{openquestion}
An approach to resolve this question in the negative would be to construct a
continuous non--piecewise linear minimal valid function $\pi$ such that 
there exists a
minimal valid function~$\pi'\neq\pi$ with $P(\pi)\subseteq P(\pi')$ (equivalently,
$E(\pi)\subseteq E(\pi')$) that is discontinuous, and all such functions~$\pi'$ are discontinuous.  
Note that the differences 
$\bar\pi = \pi' - \pi$ cannot be effective perturbations for $\pi$, because
all effective perturbations of a continuous function~$\pi$
are Lipschitz continuous by \cite[Lemma 2.11\,(iv)]{igp_survey}.
\smallbreak

Basu et al.~\cite{basu2016structure} highlight the subspace of Lipschitz
continuous functions.  All minimal valid functions that are liftable to
cut-generating function pairs for the mixed integer problem belong to this
space \cite[Remark 2.7]{basu2016structure}.  
Define a \emph{facet in the sense of Lipschitz} to be a Lipschitz continuous
function such that $P(\pi) \subseteq P(\pi')$ implies $\pi' = \pi$ for every
Lipschitz continuous valid function $\pi'$.  Thus we can ask:

\begin{openquestion}
  Is every facet in the sense of Lipschitz a facet?
\end{openquestion}

\clearpage
\section*{Acknowledgment}
The authors wish to thank the anonymous referees for their numerous detailed
suggestions, which have been very valuable in improving the presentation of our
results.

\providecommand\ISBN{ISBN }
\bibliographystyle{../amsabbrvurl}
\bibliography{../bib/MLFCB_bib}

\clearpage
\appendix

\section{Auxiliary result}  
\label{appendix:omitted}
  
In the proof of \autoref{lemma:lipschitz-equiv-perturbation}, we need the
following elementary geometric estimate. 
\begin{lemma}
\label{lemma:affine-function-min-value}
Let $F \subset [0,1]^2$ be a convex polygon with vertex set $\verts(F)$, and let $g\colon F \to \R$ be an affine linear function. Suppose that for each $v \in \verts(F)$, either $g(v) = 0$ or $g(v) \geq m$ for some $m > 0$. Let $S = \{\,x \in F\st g(x) = 0\,\}$, and assume that $S$ is nonempty. Then $g(x) \geq m\,d(x, S)/2$ for any $x \in F$, where $d(x, S)$ denotes the Euclidean distance from $x$ to $S$.
\end{lemma}
\begin{proof}
Let $x \in F$ be arbitrary. We may write
\[ x = \sum_{v \in \verts(F)} \alpha_v v \]
for some $\alpha_v \in [0,1]$ with $\sum_v \alpha_v = 1$. 

Since $S$ is a closed set, for each $v \in \verts(F)$, there exists $s_v \in S$ such that $d(v,S)=d(v, s_v)$. 
Let $s^* = \sum_{v \in \verts(F)} \alpha_v s_v$. We have that $s^* \in S$ since the set $S$ is convex.
Thus,
\begin{align*}
d(x, S) & \leq d(x, s^*) & \text{ (by definition)} \\
& = d(\,\sum_{v} \alpha_v v, \sum_{v} \alpha_v s_v \,)&\\
& \leq  \sum_{v} \alpha_v d(v, s_v) \quad \quad &\text{ (by the triangle inequality)} \\
& =  \sum_{v}  \alpha_v d(v, S).&
\end{align*}
For those $v \in \verts(F)$ with $g(v) = 0$, we have $v \in S$ by definition and thus $d(v,S) = 0$. Therefore,
\[ d(x, S) \leq \sum_{\substack{v \in \verts(F) \\ g(v) \geq m}} \alpha_v d(v, S) \leq 2 \sum_{\substack{v \in \verts(F) \\ g(v) \geq m}} \alpha_v . \]
Using the affine linearity of $g$, it thus follows that
\[ g(x) = \sum_{v \in \verts(F)} \alpha_v g(v) = \sum_{\substack{v \in \verts(F) \\ g(v) = 0}} \alpha_v g(v) + \sum_{\substack{v \in \verts(F) \\ g(v) \geq m}} \alpha_v g(v) \geq \frac{m\,d(x, S)}{2}.\]
\end{proof}

\clearpage



\section{Data of the function $\pi=\sage{kzh\_minimal\_has\_only\_crazy\_perturbation\_1()}$}
\label{s:kzh_minimal_has_only_crazy_perturbation_1}

The following pages provide tables with data of the piecewise
linear function
$\pi=\sage{kzh\_minimal\_has\_only\_crazy\_perturbation\_1()}$ of \autoref{lemma:discontinuous_examples_2}.

\autoref{tab:kzh_minimal_has_only_crazy_perturbation_1} defines the function
by listing the breakpoints $x_i$ and the values and the left and right limits
at the breakpoints. (A version of this table has previously appeared in
\cite{koeppe-zhou:crazy-perturbation}.)

Tables~\ref{tab:kzh_minimal_has_only_crazy_perturbation_1_faces_used_dim_2}
and \ref{tab:kzh_minimal_has_only_crazy_perturbation_1_faces_used_dim_1} list
the faces $F = F(I, J, K)$ of the complex~$\Delta\P$ 
that we use for proving piecewise linearity of $\bar\pi$ outside of the special intervals, i.e., Claim~(o) in the proof of \autoref{lemma:discontinuous_examples_2}.
In all tables, the faces are listed by lexicographically increasing triples
$(I, J, K)$; and of the two equivalent faces $F(I, J, K)$ and $F(J, I, K)$, we
only show the lexicographically smaller one.

\autoref{tab:kzh_minimal_has_only_crazy_perturbation_1_faces_of_vertices_used}
shows a list of faces~$F$ that satisfy $\Delta\bar\pi_F(x,y)=0$ for all
$(x, y) \in \relint(F)$.  This property can be verified by inspecting the
provided list of vertices of each face.  A selection of one vertex $(u, v)$
for each listed face~$F$, listed first in the table, suffices to form a
full-rank homogeneous 
linear system of equations $\Delta\bar\pi_F(u, v) = 0$.  We obtained the
selection of faces and their vertices by Gaussian elimination.
The full-rank system, shown in \autoref{tab:kzh_minimal_has_only_crazy_perturbation_finite_system}, proves that $\bar\pi$ is~0 outside of the special intervals, Claim~(\ref{li:bar-pi-zero-outside-special}) in the proof of \autoref{lemma:discontinuous_examples_2}.

Finally, Tables~\ref{tab:kzh_minimal_has_only_crazy_perturbation_1_delta_pi_dim_1} and
\ref{tab:kzh_minimal_has_only_crazy_perturbation_1_delta_pi_dim_2} list the
faces $F$ whose projections $p_i(\relint(F))$, $i=1,2,3$, overlap with the
special intervals ($n_F>0$). 
They are relevant for verifying Claims~(\ref{li:bar-pi-constant-on-cosets-on-special}), (\ref{li:bar-pi-additivities-on-special-at-z=l+u--pre}), and~(\ref{claim:n_F}) in the proof
of~\autoref{lemma:discontinuous_examples_2}.  
For each face $F$, we list the
values of the subadditivity slack $\Delta\pi_F(u, v)$ for all vertices $(u,v)$ of $F$ in nondecreasing
order from left to right.  
If there is an enclosing face $F'\supset F$ with $\Delta\pi_{F'}(u, v) =
\Delta\pi_F(u,v)$ for all vertices $(u,v)$ of $F$ because of one-sided continuity, then we
suppress $F$ in the table.  
All numbers have been rounded to 3 decimals for
presentation.  Claims~(\ref{li:bar-pi-constant-on-cosets-on-special}) and (\ref{li:bar-pi-additivities-on-special-at-z=l+u--pre}) 
use faces with $\Delta\pi_F(x,y) = 0$ for $(x, y)\in F$. 
To verify Claim~(\ref{li:bar-pi-abs-leq-s-on-special})
, note that if $\Delta\pi_F(u,v)=0$ 
for one vertex of~$F$, then $\Delta\pi_F(u,v)=0$ for all vertices of~$F$.
Next, note that for all other faces~$F$ with $n_F>0$, the inequality
$\Delta\pi_F(u,v) \geq n_F\cdot s$ (where $s \approx 0.001$) is satisfied and
tight for at most one vertex $(u, v)$ of each face. These vertices are marked
by the word ``(tight)'' 
in the tables; we have $n_F=1$ for each of these
faces.  All remaining subadditivity slacks $\Delta\pi_F(u,v)$ for vertices
$(u, v)\in\verts(F)$ exceed $0.003 \geq 3\cdot s$.

\clearpage
\setlength{\LTcapwidth}{0.8\textheight}

\DIFFPROTECT{
\iftrue
\newcommand\piximinus{\pi(x_i^-) = \pi_{[x_{i-1}, x_i]}(x_i)}
\newcommand\pixiplus{\pi(x_i^+)  = \pi_{[x_i, x_{i+1}]}(x_i)}
\begin{landscape}
\begin{table}[p]
  \caption{The piecewise linear function
    $\pi = \sage{kzh\_minimal\_has\_only\_crazy\_perturbation\_1}()$, defined by its
    values and limits at the breakpoints.  If a limit is omitted, it equals
    the value.}
  \label{tab:kzh_minimal_has_only_crazy_perturbation_1}
  \centering\small
  \def\arraystretch{1.33}
$\begin{array}{@{\quad}c@{\qquad\qquad}*5c}
  \toprule
  i & x_i & \piximinus & \pi(x_i) & \pixiplus & \text{slope}\\
  \midrule
  0 & 0 & \frac{101}{650} & 0 & \frac{101}{650} & \smash{\raisebox{-1.5ex}{$c_3 = -5$}}\\
  1 & \frac{101}{5000} & \frac{707}{13000} & \frac{2727}{13000} & \frac{707}{13000} & \smash{\raisebox{-1.5ex}{$c_1 = \frac{35}{13}$}}\\
  2 & \frac{60153}{369200} &  & \frac{421071}{959920} &  & \smash{\raisebox{-1.5ex}{$c_3 = -5$}}\\
  3 & \frac{849}{5000} & \frac{4851099}{11999000} & -\frac{1925}{71994} \sqrt{2} + \frac{4851099}{11999000} & \frac{4851099}{11999000} & \smash{\raisebox{-1.5ex}{$c_1 = \frac{35}{13}$}}\\
  4 & \frac{1925}{298129} \sqrt{2} + \frac{849}{5000} &  & \frac{67375}{3875677} \sqrt{2} + \frac{4851099}{11999000} &  & \smash{\raisebox{-1.5ex}{$c_3 = -5$}}\\
  5 & \frac{77}{7752} \sqrt{2} + \frac{849}{5000} & \frac{385}{93016248} \sqrt{2} + \frac{4851099}{11999000} & \frac{2695}{100776} \sqrt{2} + \frac{4851099}{11999000} & \frac{385}{93016248} \sqrt{2} + \frac{4851099}{11999000} & \smash{\raisebox{-1.5ex}{$c_1 = \frac{35}{13}$}}\\
  6 & \llap{$a_0 = {}$}\frac{19}{100} & -\frac{1925}{71994} \sqrt{2} + \frac{275183}{599950} & \frac{18196}{59995} & -\frac{1925}{71994} \sqrt{2} + \frac{275183}{599950} & \smash{\raisebox{-1.5ex}{$c_1 = \frac{35}{13}$}}\\
  7 & \frac{77}{22152} \sqrt{2} + \frac{281986521}{1490645000} &  & -\frac{385}{22152} \sqrt{2} + \frac{10467633}{22933000} &  & \smash{\raisebox{-1.5ex}{$c_3 = -5$}}\\
  8 & \frac{40294}{201875} & \frac{848837}{2099500} & \frac{795836841}{1937838500} & \frac{848837}{2099500} & \smash{\raisebox{-1.5ex}{$c_1 = \frac{35}{13}$}}\\
  9 & \frac{36999}{184600} &  & \frac{975607}{2399800} &  & \smash{\raisebox{-1.5ex}{$c_3 = -5$}}\\
  10 & \llap{$a_1 = {}$}\frac{77}{7752} \sqrt{2} + \frac{19}{100} & -\frac{385}{7752} \sqrt{2} + \frac{275183}{599950} & \frac{385}{93016248} \sqrt{2} + \frac{18196}{59995} & -\frac{385}{7752} \sqrt{2} + \frac{275183}{599950} & \smash{\raisebox{-1.5ex}{$c_3 = -5$}}\\
  11 & \frac{1051}{5000} & \frac{4291761}{11999000} & -\frac{1925}{71994} \sqrt{2} + \frac{4291761}{11999000} & \frac{4291761}{11999000} & \smash{\raisebox{-1.5ex}{$c_1 = \frac{35}{13}$}}\\
  12 & \frac{1925}{298129} \sqrt{2} + \frac{1051}{5000} &  & \frac{67375}{3875677} \sqrt{2} + \frac{4291761}{11999000} &  & \smash{\raisebox{-1.5ex}{$c_3 = -5$}}\\
  13 & \llap{$a_2 = {}$}\frac{14199}{64600} & \frac{192500}{3875677} \sqrt{2} + \frac{240046061}{775135400} & \frac{50943}{167960} & \frac{192500}{3875677} \sqrt{2} + \frac{240046061}{775135400} & \smash{\raisebox{-1.5ex}{$c_3 = -5$}}\\
  14 & \frac{77}{7752} \sqrt{2} + \frac{1051}{5000} & \frac{385}{93016248} \sqrt{2} + \frac{4291761}{11999000} & \frac{2695}{100776} \sqrt{2} + \frac{4291761}{11999000} & \frac{385}{93016248} \sqrt{2} + \frac{4291761}{11999000} & \smash{\raisebox{-1.5ex}{$c_1 = \frac{35}{13}$}}\\
  15 & \frac{77}{22152} \sqrt{2} + \frac{342208579}{1490645000} &  & -\frac{385}{22152} \sqrt{2} + \frac{122181831}{298129000} &  & \smash{\raisebox{-1.5ex}{$c_3 = -5$}}\\
  16 & \frac{193799}{807500} &  & \frac{187742}{524875} &  & \smash{\raisebox{-1.5ex}{$c_1 = \frac{35}{13}$}}\\
  17 & \llap{$l = A = {}$}\frac{219}{800} &  & \frac{933}{2080} & \frac{51443}{147680} & \smash{\raisebox{-1.5ex}{$c_2 = \frac{5}{11999}$}}\\
  18 & \llap{$u = A_0 = {}$}\frac{269}{800} & \frac{668809}{1919840} & \frac{683}{2080} &  & \smash{\raisebox{-1.5ex}{$c_1 = \frac{35}{13}$}}\\
  19 & \llap{$f - u = {}$}\frac{371}{800} &  & \frac{1397}{2080} & \frac{1251031}{1919840} & \smash{\raisebox{-1.5ex}{$c_2 = \frac{5}{11999}$}}\\
  20 & \llap{$f - l = {}$}\frac{421}{800} & \frac{96237}{147680} &
                                                                   \frac{1147}{2080} &  & \smash{\raisebox{-1.5ex}{$c_1 = \frac{35}{13}$}}\\
\end{array}$
\end{table}
\begin{table}[p]
  \centering\small
  \def\arraystretch{1.33}
$\begin{array}{@{\quad}c@{\qquad\qquad}*5c}
  &&&&  & \smash{\raisebox{-1.5ex}{$c_1 = \frac{35}{13}$}}\\
  21 & \frac{452201}{807500} &  & \frac{337133}{524875} &  & \smash{\raisebox{-1.5ex}{$c_3 = -5$}}\\
  22 & -\frac{77}{22152} \sqrt{2} + \frac{850307421}{1490645000} &  & \frac{385}{22152} \sqrt{2} + \frac{175947169}{298129000} &  & \smash{\raisebox{-1.5ex}{$c_1 = \frac{35}{13}$}}\\
  23 & -\frac{77}{7752} \sqrt{2} + \frac{2949}{5000} & -\frac{385}{93016248} \sqrt{2} + \frac{7707239}{11999000} & -\frac{2695}{100776} \sqrt{2} + \frac{7707239}{11999000} & -\frac{385}{93016248} \sqrt{2} + \frac{7707239}{11999000} & \smash{\raisebox{-1.5ex}{$c_3 = -5$}}\\
  24 & \llap{$f - a_2 = {}$}\frac{37481}{64600} & -\frac{192500}{3875677} \sqrt{2} + \frac{535089339}{775135400} & \frac{117017}{167960} & -\frac{192500}{3875677} \sqrt{2} + \frac{535089339}{775135400} & \smash{\raisebox{-1.5ex}{$c_3 = -5$}}\\
  25 & -\frac{1925}{298129} \sqrt{2} + \frac{2949}{5000} &  & -\frac{67375}{3875677} \sqrt{2} + \frac{7707239}{11999000} &  & \smash{\raisebox{-1.5ex}{$c_1 = \frac{35}{13}$}}\\
  26 & \frac{2949}{5000} & \frac{7707239}{11999000} & \frac{1925}{71994} \sqrt{2} + \frac{7707239}{11999000} & \frac{7707239}{11999000} & \smash{\raisebox{-1.5ex}{$c_3 = -5$}}\\
  27 & \llap{$f - a_1 = {}$}{-}\frac{77}{7752} \sqrt{2} + \frac{61}{100} & \frac{385}{7752} \sqrt{2} + \frac{324767}{599950} & -\frac{385}{93016248} \sqrt{2} + \frac{41799}{59995} & \frac{385}{7752} \sqrt{2} + \frac{324767}{599950} & \smash{\raisebox{-1.5ex}{$c_3 = -5$}}\\
  28 & \frac{110681}{184600} &  & \frac{1424193}{2399800} &  & \smash{\raisebox{-1.5ex}{$c_1 = \frac{35}{13}$}}\\
  29 & \frac{121206}{201875} & \frac{1250663}{2099500} & \frac{1142001659}{1937838500} & \frac{1250663}{2099500} & \smash{\raisebox{-1.5ex}{$c_3 = -5$}}\\
  30 & -\frac{77}{22152} \sqrt{2} + \frac{910529479}{1490645000} &  & \frac{385}{22152} \sqrt{2} + \frac{12465367}{22933000} &  & \smash{\raisebox{-1.5ex}{$c_1 = \frac{35}{13}$}}\\
  31 & \llap{$f - a_0 = l+u = {}$} \frac{61}{100} & \frac{1925}{71994} \sqrt{2} + \frac{324767}{599950} & \frac{41799}{59995} & \frac{1925}{71994} \sqrt{2} + \frac{324767}{599950} & \smash{\raisebox{-1.5ex}{$c_1 = \frac{35}{13}$}}\\
  32 & -\frac{77}{7752} \sqrt{2} + \frac{3151}{5000} & -\frac{385}{93016248} \sqrt{2} + \frac{7147901}{11999000} & -\frac{2695}{100776} \sqrt{2} + \frac{7147901}{11999000} & -\frac{385}{93016248} \sqrt{2} + \frac{7147901}{11999000} & \smash{\raisebox{-1.5ex}{$c_3 = -5$}}\\
  33 & -\frac{1925}{298129} \sqrt{2} + \frac{3151}{5000} &  & -\frac{67375}{3875677} \sqrt{2} + \frac{7147901}{11999000} &  & \smash{\raisebox{-1.5ex}{$c_1 = \frac{35}{13}$}}\\
  34 & \frac{3151}{5000} & \frac{7147901}{11999000} & \frac{1925}{71994} \sqrt{2} + \frac{7147901}{11999000} & \frac{7147901}{11999000} & \smash{\raisebox{-1.5ex}{$c_3 = -5$}}\\
  35 & \frac{235207}{369200} &  & \frac{538849}{959920} &  & \smash{\raisebox{-1.5ex}{$c_1 = \frac{35}{13}$}}\\
  36 & \frac{3899}{5000} & \frac{12293}{13000} & \frac{10273}{13000} & \frac{12293}{13000} & \smash{\raisebox{-1.5ex}{$c_3 = -5$}}\\
  37 & \llap{$f = {}$}\frac{4}{5} & \frac{549}{650} & 1 & \frac{549}{650} & \smash{\raisebox{-1.5ex}{$c_1 = \frac{35}{13}$}}\\
  38 & \frac{4101}{5000} & \frac{899}{1000} & \frac{9667}{13000} & \frac{899}{1000} & \smash{\raisebox{-1.5ex}{$c_3 = -5$}}\\
  39 & \frac{4899}{5000} & \frac{101}{1000} & \frac{3333}{13000} & \frac{101}{1000} & \smash{\raisebox{-1.5ex}{$c_1 = \frac{35}{13}$}}\\
  40 & 1 & \frac{101}{650} & 0 & \frac{101}{650} & \smash{\raisebox{-1.5ex}{$$}}\\
   \midrule
  i & x_i & \piximinus & \pi(x_i) & \pixiplus & \text{slope}\\
  \bottomrule
\end{array}$
\end{table}
\end{landscape}

\else
\begin{table}[tp]
  \caption{The piecewise linear function
    $\pi = \sage{kzh\_minimal\_has\_only\_crazy\_perturbation\_1()}$, defined by its
    values and limits at the breakpoints.  If a limit is omitted, it equals
    the value.}
  \label{tab:kzh_minimal_has_only_crazy_perturbation_1}
  \centering\small
  \def\arraystretch{1.17}
  \hspace*{-2cm}$\input{tab_kzh_minimal_has_only_crazy_perturbation_1.tex}$
\end{table}
\fi
}

\begin{landscape}
  \centering\small
\providecommand\compactop[1]{\kern0.2pt{#1}\kern0.2pt\relax}
\providecommand\specialinterval[1]{\hphantom*{#1}\text{*}}
\providecommand\tightslack[1]{\llap{$\triangleright$\,}{#1}}
\def\arraystretch{1.17}
\begin{longtable}{*{4}C*{3}{C}@{\qquad}*{3}{C}@{\qquad}*{3}{C}@{\qquad}}
\caption{Two-dimensional faces $F$ with additivity on $\intr(F)$ for proving piecewise linearity outside of the special intervals. All intervals $I$, $J$, $K$ are closed and elements of the complex~$\P$; notation $\langle a, b\rangle$: endpoints are not reached by the projection of the face; $(a, b)$: function $\pi$ is one-sided discontinuous at the endpoints from within the interval; $[a, b]$: function $\pi$ is one-sided continuous at the endpoints from within the interval.}
\label{tab:kzh_minimal_has_only_crazy_perturbation_1_faces_used_dim_2}\\
  \toprule
  \multicolumn{3}{c}{Face $F = F(I, J, K)$} &  & \multicolumn{9}{c}{vertices of $F$}\\
\cmidrule(r){1-3}\cmidrule{5-13}
  I & J & K & \text{slope} & \multicolumn{1}{C}{u} & \multicolumn{1}{C}{v} & \multicolumn{1}{C@{\qquad}}{u+v} & \multicolumn{1}{C}{u} & \multicolumn{1}{C}{v} & \multicolumn{1}{C@{\qquad}}{u+v} & \multicolumn{1}{C}{u} & \multicolumn{1}{C}{v} & \multicolumn{1}{C@{\qquad}}{u+v}\\
  \midrule
\endfirsthead
\caption{Two-dimensional faces $F$ with additivity on $\intr(F)$ for proving piecewise linearity outside of the special intervals (ctd.)}\\
  \toprule
  \multicolumn{3}{c}{Face $F = F(I, J, K)$} &  & \multicolumn{9}{c}{vertices of $F$}\\
\cmidrule(r){1-3}\cmidrule{5-13}
  I & J & K & \text{slope} & \multicolumn{1}{C}{u} & \multicolumn{1}{C}{v} & \multicolumn{1}{C@{\qquad}}{u+v} & \multicolumn{1}{C}{u} & \multicolumn{1}{C}{v} & \multicolumn{1}{C@{\qquad}}{u+v} & \multicolumn{1}{C}{u} & \multicolumn{1}{C}{v} & \multicolumn{1}{C@{\qquad}}{u+v}\\
  \midrule
\endhead
  \bottomrule
\endfoot
{}  (x_{1}, x_{2}\rangle  & (x_{1}, x_{2}\rangle  & \langle x_{1}, x_{2}] & c_{1} & x_{1}{}^+ & x_{1}{}^+ & 2 x_{1}{}^+ & x_{1}{}^+ & \compactop-x_{1} \compactop+ x_{2}{}^- & x_{2}{}^- & \compactop-x_{1} \compactop+ x_{2}{}^- & x_{1}{}^+ & x_{2}{}^-\\
{}  (x_{1}, x_{2}\rangle  & [u, f \compactop- u\rangle  & \langle u, f \compactop- u] & c_{1} & x_{1}{}^+ & u{}^+ & u \compactop+ x_{1}{}^+ & x_{1}{}^+ & f \compactop- u \compactop- x_{1}{}^- & f \compactop- u{}^- & f \compactop- 2 u{}^- & u{}^+ & f \compactop- u{}^-\\
{}  (x_{1}, x_{2}\rangle  & [x_{35}, x_{36}\rangle  & \langle x_{35}, x_{36}) & c_{1} & x_{1}{}^+ & x_{35}{}^+ & x_{1} \compactop+ x_{35}{}^+ & x_{1}{}^+ & \compactop-x_{1} \compactop+ x_{36}{}^- & x_{36}{}^- & \compactop-x_{35} \compactop+ x_{36}{}^- & x_{35}{}^+ & x_{36}{}^-\\
{}  \langle x_{38}, x_{39}) & \langle x_{38}, x_{39}) & (x_{38} \compactop+ 1, x_{39} \compactop+ 1\rangle  & c_{3} & x_{38} \compactop- x_{39} \compactop+ 1{}^+ & x_{39}{}^- & x_{38} \compactop+ 1{}^+ & x_{39}{}^- & x_{38} \compactop- x_{39} \compactop+ 1{}^+ & x_{38} \compactop+ 1{}^+ & x_{39}{}^- & x_{39}{}^- & 2 x_{39}{}^-\\
\end{longtable}
  \clearpage
\end{landscape}

\begin{landscape}
  \centering\small
\providecommand\compactop[1]{\kern0.2pt{#1}\kern0.2pt\relax}
\providecommand\specialinterval[1]{\hphantom*{#1}\text{*}}
\providecommand\tightslack[1]{\llap{$\triangleright$\,}{#1}}
\def\arraystretch{1.17}
\begin{longtable}{*{4}C*{3}{C}@{\qquad}*{3}{C}@{\qquad}}
\caption{One-dimensional faces $F$ with additivity on $\relint(F)$ for proving piecewise linearity outside of the special intervals. All intervals $I$, $J$, $K$ are closed and elements of the complex~$\P$; notation $\langle a, b\rangle$: endpoints are not reached by the projection of the face; $(a, b)$: function $\pi$ is one-sided discontinuous at the endpoints from within the interval; $[a, b]$: function $\pi$ is one-sided continuous at the endpoints from within the interval.}
\label{tab:kzh_minimal_has_only_crazy_perturbation_1_faces_used_dim_1}\\
  \toprule
  \multicolumn{3}{c}{Face $F = F(I, J, K)$} &  & \multicolumn{6}{c}{vertices of $F$}\\
\cmidrule(r){1-3}\cmidrule{5-10}
  I & J & K & \text{slope} & \multicolumn{1}{C}{u} & \multicolumn{1}{C}{v} & \multicolumn{1}{C@{\qquad}}{u+v} & \multicolumn{1}{C}{u} & \multicolumn{1}{C}{v} & \multicolumn{1}{C@{\qquad}}{u+v}\\
  \midrule
\endfirsthead
\caption{One-dimensional faces $F$ with additivity on $\relint(F)$ for proving piecewise linearity outside of the special intervals (ctd.)}\\
  \toprule
  \multicolumn{3}{c}{Face $F = F(I, J, K)$} &  & \multicolumn{6}{c}{vertices of $F$}\\
\cmidrule(r){1-3}\cmidrule{5-10}
  I & J & K & \text{slope} & \multicolumn{1}{C}{u} & \multicolumn{1}{C}{v} & \multicolumn{1}{C@{\qquad}}{u+v} & \multicolumn{1}{C}{u} & \multicolumn{1}{C}{v} & \multicolumn{1}{C@{\qquad}}{u+v}\\
  \midrule
\endhead
  \bottomrule
\endfoot
{}  \langle 0, x_{1}\rangle  & \{ a_{0} \} & [x_{9}, a_{1}) & c_{3} & \compactop-a_{0} \compactop+ x_{9}{}^+ & a_{0}{}^{\hphantom{+}} & x_{9}{}^+ & \compactop-a_{0} \compactop+ a_{1}{}^- & a_{0}{}^{\hphantom{+}} & a_{1}{}^-\\
{}  \langle 0, x_{1}) & \{ a_{0} \} & (a_{1}, x_{11}) & c_{3} & \compactop-a_{0} \compactop+ a_{1}{}^+ & a_{0}{}^{\hphantom{+}} & a_{1}{}^+ & x_{1}{}^- & a_{0}{}^{\hphantom{+}} & a_{0} \compactop+ x_{1}{}^-\\
{}  \langle 0, x_{1}\rangle  & \{ a_{1} \} & [x_{12}, a_{2}) & c_{3} & \compactop-a_{1} \compactop+ x_{12}{}^+ & a_{1}{}^{\hphantom{+}} & x_{12}{}^+ & \compactop-a_{1} \compactop+ a_{2}{}^- & a_{1}{}^{\hphantom{+}} & a_{2}{}^-\\
{}  \langle 0, x_{1}) & \{ a_{1} \} & (a_{2}, x_{14}) & c_{3} & \compactop-a_{1} \compactop+ a_{2}{}^+ & a_{1}{}^{\hphantom{+}} & a_{2}{}^+ & x_{1}{}^- & a_{1}{}^{\hphantom{+}} & a_{1} \compactop+ x_{1}{}^-\\
{}  \langle 0, x_{1}) & \{ a_{2} \} & [x_{15}, x_{16}] & c_{3} & \compactop-a_{2} \compactop+ x_{15}{}^+ & a_{2}{}^{\hphantom{+}} & x_{15}{}^+ & x_{1}{}^- & a_{2}{}^{\hphantom{+}} & a_{2} \compactop+ x_{1}{}^-\\
{}  (0, x_{1}) & \{ x_{36} \} & (x_{36}, f) & c_{3} & 0{}^+ & x_{36}{}^{\hphantom{+}} & x_{36}{}^+ & x_{1}{}^- & x_{36}{}^{\hphantom{+}} & x_{1} \compactop+ x_{36}{}^-\\
{}  (0, x_{1}) & \{ x_{38} \} & (x_{38}, x_{39}\rangle  & c_{3} & 0{}^+ & x_{38}{}^{\hphantom{+}} & x_{38}{}^+ & x_{1}{}^- & x_{38}{}^{\hphantom{+}} & x_{1} \compactop+ x_{38}{}^-\\
{}  (x_{1}, x_{2}\rangle  & \{ x_{3} \} & (a_{0}, x_{7}] & c_{1} & x_{1}{}^+ & x_{3}{}^{\hphantom{+}} & x_{1} \compactop+ x_{3}{}^+ & \compactop-x_{3} \compactop+ x_{7}{}^- & x_{3}{}^{\hphantom{+}} & x_{7}{}^-\\
\end{longtable}
  \clearpage
\end{landscape}

\begin{landscape}
  \centering\small
\providecommand\compactop[1]{\kern0.2pt{#1}\kern0.2pt\relax}
\providecommand\specialinterval[1]{\hphantom*{#1}\text{*}}
\providecommand\tightslack[1]{\llap{$\triangleright$\,}{#1}}
\def\arraystretch{1.17}
\begin{longtable}{*{3}C*{3}{>{\tiny$}c<{$}}@{\qquad}*{3}{>{\tiny$}c<{$}}@{\qquad}*{3}{>{\tiny$}c<{$}}@{\qquad}}
\caption{Faces $F$ with additivity on $\relint(F)$, one vertex of each providing an equation $\Delta\bar\pi_F(u, v)=0$, to form a full-rank homogeneous linear system in the proof of \autoref{lemma:discontinuous_examples_2}. All intervals $I$, $J$, $K$ are closed and elements of the complex~$\P$; notation $\langle a, b\rangle$: endpoints are not reached by the projection of the face; $(a, b)$: function $\pi$ is one-sided discontinuous at the endpoints from within the interval; $[a, b]$: function $\pi$ is one-sided continuous at the endpoints from within the interval.}
\label{tab:kzh_minimal_has_only_crazy_perturbation_1_faces_of_vertices_used}\\
  \toprule
  \multicolumn{3}{c}{Face $F = F(I, J, K)$} & \multicolumn{3}{c}{selected vertex} & \multicolumn{6}{c}{other vertices of $F$}\\
\cmidrule(r){1-3}\cmidrule(r){4-6}\cmidrule{7-12}
  I & J & K & \multicolumn{1}{C}{u} & \multicolumn{1}{C}{v} & \multicolumn{1}{C@{\qquad}}{u+v} & \multicolumn{1}{C}{u} & \multicolumn{1}{C}{v} & \multicolumn{1}{C@{\qquad}}{u+v} & \multicolumn{1}{C}{u} & \multicolumn{1}{C}{v} & \multicolumn{1}{C@{\qquad}}{u+v}\\
  \midrule
\endfirsthead
\caption{Faces $F$ with additivity on $\relint(F)$, one vertex of each providing an equation $\Delta\bar\pi_F(u, v)=0$ (ctd.)}\\
  \toprule
  \multicolumn{3}{c}{Face $F = F(I, J, K)$} & \multicolumn{3}{c}{selected vertex} & \multicolumn{6}{c}{other vertices of $F$}\\
\cmidrule(r){1-3}\cmidrule(r){4-6}\cmidrule{7-12}
  I & J & K & \multicolumn{1}{C}{u} & \multicolumn{1}{C}{v} & \multicolumn{1}{C@{\qquad}}{u+v} & \multicolumn{1}{C}{u} & \multicolumn{1}{C}{v} & \multicolumn{1}{C@{\qquad}}{u+v} & \multicolumn{1}{C}{u} & \multicolumn{1}{C}{v} & \multicolumn{1}{C@{\qquad}}{u+v}\\
  \midrule
\endhead
  \bottomrule
\endfoot
{}  \langle 0, x_{1}\rangle  & \{ a_{0} \} & \{ x_{8} \} & \compactop-a_{0} \compactop+ x_{8}{}^{\hphantom{+}} & a_{0}{}^{\hphantom{+}} & x_{8}{}^{\hphantom{+}}\\
{}  \langle 0, x_{1}\rangle  & \{ a_{0} \} & [x_{9}, a_{1}) & \compactop-a_{0} \compactop+ x_{9}{}^+ & a_{0}{}^{\hphantom{+}} & x_{9}{}^+ & \compactop-a_{0} \compactop+ a_{1}{}^- & a_{0}{}^- & a_{1}{}^-\\
{}  \langle 0, x_{1}) & \{ a_{0} \} & (a_{1}, x_{11}) & x_{1}{}^- & a_{0}{}^{\hphantom{+}} & a_{0} \compactop+ x_{1}{}^- & \compactop-a_{0} \compactop+ a_{1}{}^+ & a_{0}{}^- & a_{1}{}^+\\
{}  \langle 0, x_{1}\rangle  & \{ a_{1} \} & [x_{12}, a_{2}) & \compactop-a_{1} \compactop+ x_{12}{}^+ & a_{1}{}^{\hphantom{+}} & x_{12}{}^+ & \compactop-a_{1} \compactop+ a_{2}{}^- & a_{1}{}^- & a_{2}{}^-\\
{}  \langle 0, x_{1}) & \{ a_{1} \} & (a_{2}, x_{14}) & x_{1}{}^- & a_{1}{}^{\hphantom{+}} & a_{1} \compactop+ x_{1}{}^- & \compactop-a_{1} \compactop+ a_{2}{}^+ & a_{1}{}^- & a_{2}{}^+\\
{}  \langle 0, x_{1}\rangle  & \{ a_{2} \} & \{ x_{15} \} & \compactop-a_{2} \compactop+ x_{15}{}^{\hphantom{+}} & a_{2}{}^{\hphantom{+}} & x_{15}{}^{\hphantom{+}}\\
{}  \langle 0, x_{1}) & \{ a_{2} \} & [x_{15}, x_{16}] & x_{1}{}^- & a_{2}{}^{\hphantom{+}} & a_{2} \compactop+ x_{1}{}^- & \compactop-a_{2} \compactop+ x_{15}{}^+ & a_{2}{}^- & x_{15}{}^+\\
{}  (0, x_{1}) & \{ x_{36} \} & (x_{36}, f) & 0{}^+ & x_{36}{}^{\hphantom{+}} & x_{36}{}^+ & x_{1}{}^- & x_{36}{}^- & x_{1} \compactop+ x_{36}{}^-\\
\end{longtable}
  \clearpage
\end{landscape}

\begin{landscape}
  \footnotesize
\def\arraystretch{1.17}
\begin{longtable}{*3{>{\tiny$}c<{$}@{}}@{\quad}*{35}{>{\tiny$}c<{$}@{}}@{\quad}*{2}{>{\tiny$}c<{$}@{}}@{\quad}*2{>{\tiny$}c<{$}@{}}}
\caption{Homogeneous linear system for determining the restriction of $\bar\pi$ to the non-special intervals in the proof of \autoref{lemma:discontinuous_examples_2}. The variables are the values of $\bar\pi$ at breakpoints ($\bullet$) and at midpoints of intervals between breakpoints ($-$), and the slopes $\bar c_1$, $\bar c_3$ of~$\bar\pi$ on the non-special intervals.  Matrix coefficients are abbreviated as ${+} = 1$, ${-} = -1$, and ${\cdot} = 0$.}
\label{tab:kzh_minimal_has_only_crazy_perturbation_finite_system}\\
  \toprule
  \multicolumn{3}{c}{Eqn.~$\Delta\bar\pi_F(u, v)=0$} & \multicolumn{37}{c}{Coefficients of $\bar\pi(x_i)$ and $\bar\pi(\frac{x_i+x_{i+1}}2)$} & \multicolumn{2}{c}{Coefficients of slopes}\\
\cmidrule(r){1-3}\cmidrule(r){4-40}\cmidrule{41-42}
  \multicolumn{1}{C}{u} & \multicolumn{1}{C}{v} & \multicolumn{1}{C}{u+v} & - & \stackrel{x_{1} \!\!\!}{\bullet} & - & \bullet & - & \bullet & - & \bullet & - & \bullet & - & \stackrel{a_{0} \!\!\!}{\bullet} & - & \bullet & - & \bullet & - & \bullet & - & \stackrel{a_{1} \!\!\!}{\bullet} & - & \bullet & - & \bullet & - & \stackrel{a_{2} \!\!\!}{\bullet} & - & \bullet & - & \bullet & - & \bullet & - & \stackrel{l \vphantom{{}_0}}{\bullet} & \stackrel{u \vphantom{{}_0}}{\bullet}\rlap{\ \normalsize$\wr$} & - & \stackrel{x_{38} \!\!\!}{\bullet} & \multicolumn{1}{C}{\bar c_1} & \multicolumn{1}{C}{\bar c_3}\\
  \midrule
\endfirsthead
\caption{Homogeneous linear system for determining the restriction of $\bar\pi$ to the non-special intervals in the proof of \autoref{lemma:discontinuous_examples_2} (ctd.)}\\
  \toprule
  \multicolumn{3}{c}{Eqn.~$\Delta\bar\pi_F(u, v)=0$} & \multicolumn{37}{c}{Coefficients of $\bar\pi(x_i)$ and $\bar\pi(\frac{x_i+x_{i+1}}2)$} & \multicolumn{2}{c}{Coefficients of slopes}\\
\cmidrule(r){1-3}\cmidrule(r){4-40}\cmidrule{41-42}
  \multicolumn{1}{C}{u} & \multicolumn{1}{C}{v} & \multicolumn{1}{C}{u+v} & - & \stackrel{x_{1} \!\!\!}{\bullet} & - & \bullet & - & \bullet & - & \bullet & - & \bullet & - & \stackrel{a_{0} \!\!\!}{\bullet} & - & \bullet & - & \bullet & - & \bullet & - & \stackrel{a_{1} \!\!\!}{\bullet} & - & \bullet & - & \bullet & - & \stackrel{a_{2} \!\!\!}{\bullet} & - & \bullet & - & \bullet & - & \bullet & - & \stackrel{l \vphantom{{}_0}}{\bullet} & \stackrel{u \vphantom{{}_0}}{\bullet}\rlap{\ \normalsize$\wr$} & - & \stackrel{x_{38} \!\!\!}{\bullet} & \multicolumn{1}{C}{\bar c_1} & \multicolumn{1}{C}{\bar c_3}\\
  \midrule
\endhead
  \bottomrule
\endfoot
  \relax x_{36}{}^{\hphantom{+}} & 0{}^+ & x_{36}{}^+ & 2 &  -  &  \cdot  &  \cdot  &  \cdot  &  \cdot  &  \cdot  &  \cdot  &  \cdot  &  \cdot  &  \cdot  &  \cdot  &  \cdot  &  \cdot  &  \cdot  &  \cdot  &  \cdot  &  \cdot  &  \cdot  &  \cdot  &  \cdot  &  \cdot  &  \cdot  &  \cdot  &  \cdot  &  \cdot  &  \cdot  &  \cdot  &  \cdot  &  \cdot  &  \cdot  &  \cdot  &  \cdot  &  \cdot  &  \cdot  &  \cdot  &  \cdot  &  \cdot  &  \cdot \\
  \relax a_{0}{}^{\hphantom{+}} & \compactop-a_{0} \compactop+ x_{8}{}^{\hphantom{+}} & x_{8}{}^{\hphantom{+}} &  +  &  \cdot  &  \cdot  &  \cdot  &  \cdot  &  \cdot  &  \cdot  &  \cdot  &  \cdot  &  \cdot  &  \cdot  &  +  &  \cdot  &  \cdot  &  \cdot  &  -  &  \cdot  &  \cdot  &  \cdot  &  \cdot  &  \cdot  &  \cdot  &  \cdot  &  \cdot  &  \cdot  &  \cdot  &  \cdot  &  \cdot  &  \cdot  &  \cdot  &  \cdot  &  \cdot  &  \cdot  &  \cdot  &  \cdot  &  \cdot  &  \cdot  &  \cdot  & \frac12( \compactop-2 a_{0} \compactop- x_{1} \compactop+ 2 x_{8} )\\
  \relax a_{0}{}^{\hphantom{+}} & \compactop-a_{0} \compactop+ x_{9}{}^+ & x_{9}{}^+ &  +  &  \cdot  &  \cdot  &  \cdot  &  \cdot  &  \cdot  &  \cdot  &  \cdot  &  \cdot  &  \cdot  &  \cdot  &  +  &  \cdot  &  \cdot  &  \cdot  &  \cdot  &  \cdot  &  \cdot  &  -  &  \cdot  &  \cdot  &  \cdot  &  \cdot  &  \cdot  &  \cdot  &  \cdot  &  \cdot  &  \cdot  &  \cdot  &  \cdot  &  \cdot  &  \cdot  &  \cdot  &  \cdot  &  \cdot  &  \cdot  &  \cdot  &  \cdot  & \frac12( \compactop-2 a_{0} \compactop+ a_{1} \compactop- x_{1} \compactop+ x_{9} )\\
\end{longtable}
  \clearpage
\end{landscape}

\begin{landscape}
  \centering\footnotesize
\providecommand\compactop[1]{\kern0.2pt{#1}\kern0.2pt\relax}
\providecommand\specialinterval[1]{\hphantom*{#1}\text{*}}
\providecommand\tightslack[1]{\llap{$\triangleright$\,}{#1}}
\def\arraystretch{1.17}
\begin{longtable}{*{4}C{>{\tiny$}c<{$}}*{3}C@{}r}
\caption{Subadditivity slacks $\Delta\pi_F$ for $\dim F=1$ and $n_F>0$. All intervals $I$, $J$, $K$ are closed and elements of the complex~$\P$; notation $\langle a, b\rangle$: endpoints are not reached by the projection of the face; $(a, b)$: function $\pi$ is one-sided discontinuous at the endpoints from within the interval; $[a, b]$: function $\pi$ is one-sided continuous at the endpoints from within the interval. An asterisk marks the special intervals.}
\label{tab:kzh_minimal_has_only_crazy_perturbation_1_delta_pi_dim_1}\\
  \toprule
  I & J & K & n_F & \multicolumn{1}{C}{\Delta\pi_{F}(x,y), \ (x,y) \in F = F(I, J, K)} & \multicolumn{3}{L}{\Delta\pi_F(u,v), \ (u,v)\in\verts(F)} & \\
  \midrule
\endfirsthead
\caption{Subadditivity slacks $\Delta\pi_F$ for $\dim F=1$ and $n_F>0$ (ctd.)}\\
  \toprule
  I & J & K & n_F & \multicolumn{1}{C}{\Delta\pi_{F}(x,y), \ (x,y) \in F = F(I, J, K)} & \multicolumn{3}{L}{\Delta\pi_F(u,v), \ (u,v)\in\verts(F)} & \\
  \midrule
\endhead
  \bottomrule
\endfoot
{}  \{ 0 \} & \specialinterval{(l, u)} & \specialinterval{(l, u)} & 2 & 0 & 0 & 0 &  & \\
{}  \{ 0 \} & \specialinterval{(f \compactop- u, f \compactop- l)} & \specialinterval{(f \compactop- u, f \compactop- l)} & 2 & 0 & 0 & 0 &  & \\
{}  \{ x_{1} \} & \langle x_{16}, l] & \specialinterval{(l, u\rangle } & 1 & -c_{2} x + {\left(c_{1} \compactop- c_{2}\right)} y - \frac{1280651}{2999750} & 0.256 & 0.310 &  & \\
{}  \{ x_{1} \} & \specialinterval{(l, u\rangle } & \specialinterval{\langle l, u)} & 2 & -c_{2} x + \frac{2727}{13000} & 0.210 & 0.210 &  & \\
{}  \{ x_{1} \} & \specialinterval{\langle l, u)} & [u, f \compactop- u\rangle  & 1 & -c_{1} x - {\left(c_{1} \compactop- c_{2}\right)} y + \frac{1702237}{1499875} & 0.175 & 0.230 &  & \\
{}  \{ x_{1} \} & \langle u, f \compactop- u] & \specialinterval{(f \compactop- u, f \compactop- l\rangle } & 1 & -c_{2} x + {\left(c_{1} \compactop- c_{2}\right)} y - \frac{1527763}{1499875} & 0.175 & 0.230 &  & \\
{}  \{ x_{1} \} & \specialinterval{(f \compactop- u, f \compactop- l\rangle } & \specialinterval{\langle f \compactop- u, f \compactop- l)} & 2 & -c_{2} x + \frac{2727}{13000} & 0.210 & 0.210 &  & \\
{}  \{ x_{1} \} & \specialinterval{\langle f \compactop- u, f \compactop- l)} & [f \compactop- l, x_{21}\rangle  & 1 & -c_{1} x - {\left(c_{1} \compactop- c_{2}\right)} y + \frac{5179349}{2999750} & 0.256 & 0.310 &  & \\
\end{longtable}\def\arraystretch{1.17}
\begin{longtable}{*{4}C{>{\tiny$}c<{$}}*{6}C@{}r}
\caption{Subadditivity slacks $\Delta\pi_F$ for $\dim F=2$ and $n_F>0$. All intervals $I$, $J$, $K$ are closed and elements of the complex~$\P$; notation $\langle a, b\rangle$: endpoints are not reached by the projection of the face; $(a, b)$: function $\pi$ is one-sided discontinuous at the endpoints from within the interval; $[a, b]$: function $\pi$ is one-sided continuous at the endpoints from within the interval. An asterisk marks the special intervals.}
\label{tab:kzh_minimal_has_only_crazy_perturbation_1_delta_pi_dim_2}\\
  \toprule
  I & J & K & n_F & \multicolumn{1}{C}{\Delta\pi_{F}(x,y), \ (x,y) \in F = F(I, J, K)} & \multicolumn{6}{L}{\Delta\pi_F(u,v), \ (u,v)\in\verts(F)} & \\
  \midrule
\endfirsthead
\caption{Subadditivity slacks $\Delta\pi_F$ for $\dim F=2$ and $n_F>0$ (ctd.)}\\
  \toprule
  I & J & K & n_F & \multicolumn{1}{C}{\Delta\pi_{F}(x,y), \ (x,y) \in F = F(I, J, K)} & \multicolumn{6}{L}{\Delta\pi_F(u,v), \ (u,v)\in\verts(F)} & \\
  \midrule
\endhead
  \bottomrule
\endfoot
{}  (0, x_{1}) & \langle x_{16}, l] & \specialinterval{(l, u\rangle } & 1 & -{\left(c_{2} \compactop- c_{3}\right)} x + {\left(c_{1} \compactop- c_{2}\right)} y - \frac{1155033}{2399800} & 0.100 & 0.155 & 0.256 &  &  &  & \\
{}  (0, x_{1}) & \specialinterval{(l, u)} & \specialinterval{(l, u)} & 2 & -{\left(c_{2} \compactop- c_{3}\right)} x + \frac{101}{650} & 0.054 & 0.054 & 0.155 & 0.155 &  &  & \\
{}  (0, x_{1}) & \specialinterval{\langle l, u)} & [u, f \compactop- u\rangle  & 1 & -{\left(c_{1} \compactop- c_{3}\right)} x - {\left(c_{1} \compactop- c_{2}\right)} y + \frac{2593067}{2399800} & 0.020 & 0.074 & 0.175 &  &  &  & \\
{}  (0, x_{1}) & \langle u, f \compactop- u] & \specialinterval{(f \compactop- u, f \compactop- l\rangle } & 1 & -{\left(c_{2} \compactop- c_{3}\right)} x + {\left(c_{1} \compactop- c_{2}\right)} y - \frac{2574933}{2399800} & 0.020 & 0.074 & 0.175 &  &  &  & \\
{}  (0, x_{1}) & \specialinterval{(f \compactop- u, f \compactop- l)} & \specialinterval{(f \compactop- u, f \compactop- l)} & 2 & -{\left(c_{2} \compactop- c_{3}\right)} x + \frac{101}{650} & 0.054 & 0.054 & 0.155 & 0.155 &  &  & \\
{}  (0, x_{1}) & \specialinterval{\langle f \compactop- u, f \compactop- l)} & [f \compactop- l, x_{21}\rangle  & 1 & -{\left(c_{1} \compactop- c_{3}\right)} x - {\left(c_{1} \compactop- c_{2}\right)} y + \frac{4012967}{2399800} & 0.100 & 0.155 & 0.256 &  &  &  & \\
{}  \langle x_{1}, x_{2}] & \langle x_{1}, x_{2}] & \specialinterval{(l, u\rangle } & 1 & {\left(c_{1} \compactop- c_{2}\right)} x + {\left(c_{1} \compactop- c_{2}\right)} y - \frac{33427}{95992} & 0.389 & 0.389 & 0.529 &  &  &  & \\
{}  \langle x_{1}, x_{2}] & [x_{2}, x_{3}) & \specialinterval{(l, u\rangle } & 1 & {\left(c_{1} \compactop- c_{2}\right)} x - {\left(c_{2} \compactop- c_{3}\right)} y + \frac{6683}{7384} & 0.336 & 0.389 & 0.495 & 0.529 &  &  & \\
\end{longtable}
  \clearpage
\end{landscape}

\end{document}